\documentclass[11pt,twoside,a4paper]{amsart}

\usepackage[utf8]{inputenc}

\usepackage{amssymb}
\usepackage{enumerate}
\usepackage[colorlinks,citecolor=blue,urlcolor=blue]{hyperref}
\usepackage[english]{babel}

\newtheorem{proposition}{Proposition}
\newtheorem{theorem}{Theorem}
\newtheorem{lemma}{Lemma}
\newtheorem{corollary}{Corollary}

\theoremstyle{definition}
\newtheorem{definition}{Definition}

\theoremstyle{remark}

\newtheorem{example}{Example}

\DeclareMathOperator{\Dom}{Dom}
\DeclareMathOperator{\discr}{discr}
\DeclareMathOperator{\tridiag}{tridiag}

\DeclareMathOperator{\Bor}{Bor}
\DeclareMathOperator{\tr}{tr}

\newcommand{\NN}{\mathbb{N}}

\newcommand{\RR}{\mathbb{R}}
\newcommand{\ZZ}{\mathbb{Z}}
\newcommand{\CC}{\mathbb{C}}

\newcommand{\calA}{\mathcal{A}}
\newcommand{\calV}{\mathcal{V}}
\newcommand{\calB}{\mathcal{B}}

\newcommand{\calC}{\mathcal{C}}
\newcommand{\calF}{\mathcal{F}}

\newcommand{\calX}{\mathcal{X}}
\newcommand{\calD}{\mathcal{D}}
\newcommand{\calL}{\mathcal{L}}

\newcommand{\Id}{\mathrm{Id}}

\renewcommand{\restriction}{\mathord{\upharpoonright}}

\newcommand{\sprod}[2]{\langle {#1}, {#2} \rangle}
\newcommand{\norm}[1]{\lVert {#1} \rVert}
\newcommand{\normM}[1]{\lVert {#1} \rVert_\infty}
\newcommand{\abs}[1]{\lvert {#1} \rvert}
\newcommand{\sym}[1]{\mathrm{sym} \left[ {#1} \right] }
\newcommand{\floor}[1]{{\lfloor #1 \rfloor}}

\newcommand{\ud}{\mathrm{d}}
\newcommand{\Mod}[1]{\ (\text{mod}\ #1)}
\newcommand{\per}{\mathrm{per}}

\newcommand{\sigmaEss}[1]{\sigma_{\mathrm{ess}}(#1)}
\newcommand{\sigmaP}[1]{\sigma_{\mathrm{p}}(#1)}

\theoremstyle{plain}
\newcounter{thm}

\newtheorem{main_theorem}[thm]{Theorem}

\author{Grzegorz Świderski}
\email{gswider@math.uni.wroc.pl}
\address{
	Instytut Matematyczny\\
	Uniwersytet Wrocławski\\
	Pl. Grunwaldzki 2/4\\
	50-384 Wrocław\\
	Poland}

\title[Periodic perturbations of Jacobi matrices]{Periodic perturbations of unbounded\\Jacobi matrices III: The soft edge regime}

\keywords{Jacobi matrix, orthogonal polynomials, compact resolvent, spectral gap, absolute continuity, Turán determinants}

\subjclass[2010]{Primary: 47B25, 47B36, 42C05.}

\begin{document}
\selectlanguage{english}

\begin{abstract}
   We present pretty detailed spectral analysis of Jacobi matrices with periodically modulated entries
   in the case when $0$ lies on the soft edge of the spectrum of the corresponding periodic Jacobi matrix. 
   In particular, we show that the studied operators are always self-adjoint irrespective of the
   modulated sequence. Moreover, if the growth of the modulated sequence is superlinear, then 
   the spectrum of the considered operators is always discrete. Finally, we study regular perturbations of 
   this class in the linear and the sublinear cases. We impose conditions assuring that the spectrum 
   is absolute continuous on some regions of the real line. A constructive formula for the density in terms
   of Turán determinants is also provided.
\end{abstract}

\maketitle

\section{Introduction}
   Consider two sequences $a = (a_n \colon n \geq 0)$ and $b = (b_n \colon n \geq 0)$ such that for every 
   $n \geq 0$ one has $a_n > 0$ and $b_n \in \RR$.
   Then one defines the symmetric tridiagonal matrix by the formula
   \begin{equation*}
      \calA =
      \begin{pmatrix}
         b_0 & a_0 & 0   & 0      &\ldots \\
         a_0 & b_1 & a_1 & 0       & \ldots \\
         0   & a_1 & b_2 & a_2     & \ldots \\
         0   & 0   & a_2 & b_3   &  \\
         \vdots & \vdots & \vdots  &  & \ddots
      \end{pmatrix}.
   \end{equation*}
   The action of $\calA$ on any sequence is defined by the formal matrix multiplication. Let the operator 
   $A$ be the restriction of $\calA$ to $\ell^2$, i.e. 
   $\Dom(A) = \{ x \in \ell^2 \colon \calA x \in \ell^2 \}$ and $A x = \calA x$ for $x \in \Dom(A)$, where
   \[
      \sprod{x}{y}_{\ell^2} = \sum_{n=0}^\infty x_n \overline{y_n}, \quad \ell^2 = 
      \{ x \in \CC^\NN \colon \sprod{x}{x}_{\ell^2} < \infty \}.
   \]
   The operator $A$ is called \emph{Jacobi matrix}. It is self-adjoint provided Carleman condition 
   is satisfied, i.e. 
   \begin{equation} \label{eq:52}
   	\sum_{n=0}^\infty \frac{1}{a_n} = \infty.
   \end{equation}
   
	A generalised eigenvector $u$ associated with $x \in \RR$ is any sequence satisfying the recurrence
	relation
   \begin{equation} \label{eq:53}
		a_{n-1} u_{n-1} + b_n u_n + a_n u_{n+1} = x u_n \quad (n \geq 1).
   \end{equation}
   The most important one is the sequence $(p_n(x) : n \geq 0)$ satisfying \eqref{eq:53}
   with the initial conditions: $p_{0}(x) = 1,\ p_{1}(x) = (x-b_0)/a_0$. If operator $A$ is 
   self-adjoint, then the sequence $(p_n : n \geq 0)$ is orthonormal with respect to the scalar product
   \[
      \langle f, g \rangle_{L^2(\mu)} = \int_{\mathbb{R}} f(x) \overline{g(x)} \ud \mu(x), \quad 
      L^2(\mu) = \{ f \colon \langle f, f \rangle_{L^2(\mu)} < \infty \}
   \]
   for the Borel measure $\mu(\cdot) = \sprod{E_A(\cdot) \delta_0}{\delta_0}_{\ell^2}$,
   where $E_A$ is the spectral resolution of the identity for the operator $A$ and $(\delta_n : n \geq 0)$
   is the standard basis of $\ell^2$. Let $U : \ell^2 \rightarrow L^2(\mu)$ be defined by
   \[
      U \delta_n = p_n.
   \]
   Then 
   \[
      (U A U^{-1} f) (x) = x f(x), \quad f \in L^2(\mu).
   \]
   This representation shows that the spectral analysis of $A$ is equivalent to examining the properties
   of the measure $\mu$.

	The aim of this article is to provide pretty detailed description of the spectral properties of the 
	operator~$A$ associated with sequences of the form
	\begin{equation} \label{eq:58}
		a_{k N + i} = \alpha_i \tilde{a}_k, \quad 
		b_{k N + i} = \beta_i \tilde{a}_k, \quad 
		(i=0,1,\ldots,N-1;\ k \geq 0)
	\end{equation}
	and its perturbations, where $\alpha$ and $\beta$ are $N$-periodic sequences and $\tilde{a}$ 
	is a positive sequence. According to the terminology introduced in \cite{JanasNaboko2002}, $\alpha$ 
	and $\beta$ are called \emph{modulating sequences}. It turns out that the spectral properties of $A$
	depend on the trace of the following matrix.
	\begin{equation} \label{eq:54}
		\calF(0) := 
		\prod_{i=1}^{N}
		\begin{pmatrix}
			0 & 1 \\
			-\frac{\alpha_{n-1}}{\alpha_n} & -\frac{\beta_i}{\alpha_i}
		\end{pmatrix}
	\end{equation}
	More specifically, when $\tilde{a}_k \to \infty$, and under some regularity conditions imposed on $\tilde{a}$, 
	one has that the operator $A$ has purely absolutely continuous spectrum when $|\tr \calF(0)| < 2$ and purely 
	discrete one when $|\tr \calF(0)| > 2$ (see \cite{JanasNaboko2002}). If $|\tr \calF(0)| = 2$ we have 
	two possibilities. Either the matrix $\calF(0)$ is diagonalisable (which is equivalent to $\calF(0) = \pm \Id$)
	or it is similar to some non-trivial Jordan block. In this article we are concerned with the first case, which
	has been only touched in the recent articles~\cite{Swiderski2017b, Swiderski2017}. As it turns out, 
	this class exhibits some unusual spectral properties. 

	The first theorem shows that in our setting the operator $A$ is always self-adjoint irrespectively of the 
	sequence $\tilde{a}$.
\begin{main_theorem} \label{thm:A}
	Let $N$ be a positive integer. Assume that $\calF(0) = \gamma \Id$ for some $|\gamma| = 1$, where $\calF(0)$ 
	is defined in \eqref{eq:54}. Then operator $A$ associated with the sequences defined in \eqref{eq:58} 
	is \emph{always} self-adjoint 
	and\footnote{By $\sigma(A), \sigmaP{A}$ and $\sigmaEss{A}$ we denote: the spectrum, the point 
	spectrum and the essential spectrum of the operator $A$, respectively.} $0 \notin \sigmaP{A}$.
\end{main_theorem}
Let us recall that the operator $A$ is self-adjoint if and only if at least one non-zero solution
of \eqref{eq:53} for $x=0$ does not belong to $\ell^2$. In the case $b_n \equiv 0$ the recurrence relation 
for $x=0$ can be solved explicitly. The point of Theorem~\ref{thm:A} is that in our setting we can show 
that the sequence $(p_n(0) : n \geq 0)$ is $2N$-periodic, hence it cannot be square summable. We are
doing so by the use of transfer matrices. It turns out that the condition~\eqref{eq:54} implies that
$(p_n(0) : n \geq 0)$ does not depend on the sequence $\tilde{a}$.

The conclusion of Theorem~\ref{thm:A} is in a sharp contrast with the case when $|\tr \calF(0)| < 2$. 
In such a case it can be shown that under some regularity conditions imposed on $\tilde{a}$, operator 
$A$ is self-adjoint if and only if Carleman condition is satisfied (see \cite{Swiderski2017}). 
In Section~\ref{sec:constrModulSeq} we prove that the class of modulating sequences satisfying 
$\calF(0) = \pm \Id$ is surprisingly rich, we explain its geometric meaning and we provide some specific examples.

In the next theorem we study the spectral properties of $A$ in more detail.
\begin{main_theorem} \label{thm:B}
	Let the assumptions of Theorem~\ref{thm:A} be satisfied. Then we have the following cases.
	\begin{enumerate}[(a)]
		\item If 
		$\begin{aligned}[b]
			\lim_{k \rightarrow \infty} \frac{k}{\tilde{a}_k} = \infty
		\end{aligned}$, 
		then $0 \in \sigmaEss{A}$. \label{thm:B:eq:2}
		
		\item If the sequence 
		$\begin{aligned}[b]
			\left( \frac{k}{\tilde{a}_k} : k \geq 1 \right)
		\end{aligned}$
		is bounded, then $0 \notin \sigmaEss{A}$. \label{thm:B:eq:3}
		
		\item If
		$\begin{aligned}[b]
			\lim_{k \rightarrow \infty} \frac{k}{\tilde{a}_k} = 0
		\end{aligned}$,
		then $\sigmaEss{A} = \emptyset$. \label{thm:B:eq:4}
	\end{enumerate}
\end{main_theorem}
The conclusions of Theorem~\ref{thm:B} have been known only in the setting of $\alpha_n \equiv 1$ and
$\beta_n \equiv 0$. Specifically, the case \eqref{thm:B:eq:2} has been proven in
\cite{DombrowskiPedersen1995}, the case \eqref{thm:B:eq:3} in \cite{Janas2007, Moszynski2003, Sahbani2008}
and the case \eqref{thm:B:eq:4} in \cite{Janas2007, HintonLewis1978, Sahbani2008}.

It turns out that $\calA$ associated with the sequences defined in \eqref{eq:58} has simple representation
in the form of a block Jacobi matrix. This connection allows us to derive explicit formula for 
$\calA^{-1}$. In doing so the properties of orthonormal polynomials are crucial. In fact $\calA^{-1}$
may be expressed as the symmetrisation of a lower triangular Toeplitz matrix multiplied by a diagonal one.
Then the proof of Theorem~\ref{thm:B} contained in Section~\ref{sec:spectPhaseTrans}
is reduced to determining when $\calA^{-1}$ restricted to $\ell^2$ is an unbounded, a bounded or a compact
operator. In the proof Hardy's inequality concerning boundedness of Carleman operator is useful.

As we see case \eqref{thm:B:eq:3} is the exact point when the spectral gap around $0$ in $\sigmaEss{A}$
appears. The case \eqref{thm:B:eq:4} provides pretty satisfactory understanding of the 
spectrum of $A$. Hence, in the next theorem we consider the perturbations of the remaining cases 
in more detail.
\begin{main_theorem} \label{thm:C}
	Let $N$ be a positive integer. Suppose that $(r_n : n \geq 0)$ and $(q_n : n \geq 0)$ are
	$N$-periodic sequences such that
	\begin{equation} \label{thm:C:eq:1}
		\prod_{i=0}^{N-1} 
		\begin{pmatrix}
			0 & 1 \\
			-r_i & -q_i
		\end{pmatrix} 
		=
		\gamma \Id, \quad \gamma \in \{-1, 1 \}.
	\end{equation}
	Assume
	\begin{equation} \label{thm:C:eq:2}
		\calV_N \left( \frac{1}{a_n} : n \geq 0 \right) + 
		\calV_N(r_n a_n - a_{n-1} : n \geq 1) + 
		\calV_N (b_n - q_n a_n : n \geq 0) < \infty.
	\end{equation}
	Let
	\begin{enumerate}[(a)]
		\item $\begin{aligned}[b]
			\lim_{n \rightarrow \infty} a_n = \infty,
		\end{aligned}$ \label{thm:C:eq:3}
		
		\item $\begin{aligned}[b]
			\lim_{n \rightarrow \infty} |r_n a_n - a_{n-1} - s_n| = 0,
		\end{aligned}$ \label{thm:C:eq:4}
		
		\item $\begin{aligned}[b]
			\lim_{n \rightarrow \infty} |b_n - q_n a_n - z_n| = 0
		\end{aligned}$ \label{thm:C:eq:5}
	\end{enumerate}
	for $N$-periodic sequences $s$ and $z$. Then there is an explicit computable compact interval
	$I(s,z)$	 such that for every compact set $K \subset \Lambda := \RR \setminus I(s,z)$ there are 
	constants $c_1, c_2>0$ such that for every generalised eigenvector $u$ associated with 
	$x \in K$ and every $n \geq 1$
	\[
		c_1 (u_0^2 + u_1^2) \leq a_n (u_{n-1}^2 + u_n^2) \leq c_2 (u_0^2 + u_1^2).
	\]
\end{main_theorem}
	Let us recall that the total $N$-variation of a sequence $(x_n : n \geq 0)$ is defined by
	\[
		\calV_N(x_n : n \geq 0) = \sum_{n = 0}^\infty |x_{n+N} - x_n|.
	\]
	Notice that $\calV_N(x_n : n \geq 0) < \infty$ implies that for each $j \in \{0, \ldots, N-1\}$ the
	subsequence $(x_{k N + j} : k \geq 0)$ converges.

	The asymptotics provided by Theorem~\ref{thm:C} implies that the spectrum of $A$ is purely absolutely 
	continuous on $\Lambda$ (see \cite{Clark1996}) and $\sigma(A) \supset \overline{\Lambda}$. There are
	known examples when $\sigmaEss{A} = \overline{\Lambda}$ and we expect that it is always the case. In
	Sections \ref{sec:ex:multiple}, \ref{sec:ex:mod} and \ref{sec:ex:additive} we show its special
	cases and comment the connections with earlier results known in the literature.

	In the proof of Theorem~\ref{thm:C} we analyse the asymptotic behaviour of ($N$-shifted) Turán
	determinants defined by
	\[
		D^N_n(x) = p_n(x) p_{n+N-1}(x) - p_{n-1}(x) p_{n+N}(x).
	\]
	We are doing so by careful analysis of the regularity of transfer matrices.
	The same strategy was employed in \cite{Swiderski2017} in the case of constant $r$ and $q$
	but the present case is much more involved, both for the algebraic reasons and for the lack of the 
	explicit formulas for the orthonormal polynomials associated with periodic recurrence coefficients.

	Finally, the last theorem provides the constructive formula for the density of the measure $\mu$. 
	Technically, it is a consequence of the proof of the previous theorem together with the approach
	taken in \cite{Swiderski2017b}. More precisely, we take a sequence of approximations in the weak sense
	of the measure $\mu$ and by showing the uniform convergence on every compact subset of $\Lambda$ of 
	the densities of the approximations we can identify the formula for the density of $\mu$.
\begin{main_theorem} \label{thm:D}
	Let the assumptions of Theorem~\ref{thm:C} be satisfied. Then for every $i \in \{ 0, 1, \ldots, N-1 \}$
	\[
		\widetilde{g}^i(x) = \lim_{k \rightarrow \infty} a_{(k+1)N+i-1}^2 |D_{kN+i}^N(x)|, \quad 
		x \in \Lambda
	\]	
	exists and defines a continuous positive function. Moreover, if 
	\begin{equation} \label{thm:D:eq:1}
		\lim_{n \rightarrow \infty} |a_{n+N} - a_n| = 0,
	\end{equation}
	then $\mu$ is absolutely continuous on $\Lambda$ and its density equals
	\begin{equation} \label{thm:D:eq:2}
		\mu'(x) = \frac{\sqrt{h_i(x)}}{2 \pi \widetilde{g}^i(x)}, \quad x \in \Lambda,
	\end{equation}
	where $h_i$ is an explicit polynomial of degree $2$ with positive leading coefficient.
\end{main_theorem}
In fact, $\Lambda$ is the maximal set on which $h_0$ (hence all $h_i$) is strictly positive. 
We expect that assumption \eqref{thm:D:eq:1} is not needed. 

The article is organized as follows. In Section~\ref{sec:prelim} we collect basic facts 
and the notation used in the article. Next, in Section~\ref{sec:algIdent}, we prove some
auxiliary identities for the polynomials with periodic recurrence coefficients. 
In Section~\ref{sec:selfAdj} we prove Theorem~\ref{thm:A}. In Section~\ref{sec:formRes} we provide 
the explicit formula for $\calA^{-1}$ and in Section~\ref{sec:spectPhaseTrans} we present the proof 
of Theorem~\ref{thm:B}. Sections \ref{sec:asympGenEig} and \ref{sec:formDens} are devoted to the proofs 
of Theorem \ref{thm:C} and \ref{thm:D}, respectively. Finally, in Section~\ref{sec:examples} we show 
the methods of construction of the modulating sequences and present some special cases of 
Theorem~\ref{thm:C}. 

\section{Preliminaries} \label{sec:prelim}
In this section we present some basic facts and set the notation used in the rest of the article.

\subsection{Finite matrices}
	By $M_d(\RR)$ and $M_d(\CC)$ we denote the space of $d$ by $d$ matrices with real and complex entries,
	respectively. We index them from $1$ to $d$.
	
	By $\tridiag(x_i, y_i, z_i : i \in I)$ we mean the tridiagonal matrix with values $x_i, y_i$ and $z_i$
	on the subdiagonal, the diagonal and superdiagonal, respectively. The dimension of this matrix
	is the cardinality of the indexing set $I$.
	
	For a matrix $C \in M_d(\CC)$ its real part is defined by
	\[
		\sym{C} = \frac{1}{2}(C + C^*)
	\]
	where $C^*$ is the Hermitian transpose of $C$.
	
	For a sequence of square matrices $(C_n : n \in \NN)$ and 
	$n_0, n_1 \in \NN$ we set
	\[
		\prod_{k=n_0}^{n_1} C_k = 
		\begin{cases} 
			C_{n_1} C_{n_1 - 1} \cdots C_{n_0} & n_1 \geq n_0, \\
			\Id & \text{otherwise.}
		\end{cases}
	\]
	   
	With a binary quadratic form, that is a quadratic form on $\CC^2$, represented by
	a self-adjoint matrix $C$ we associate its \emph{discriminant} given by the formula 
	$\discr(C) = (\tr C)^2 - 4 \det C$. By the invariance of the determinant and the trace on the 
	conjugation, one has
	\[
		\discr(B^{-1} C B) = \discr(C)
	\]
	consequently, for any invertible matrix $B$ one has
	\[
		\discr(B C) = \discr(B^{-1} B C B) = \discr(C B).
	\]
	Observe that for every matrix $X \in M_2(\RR)$ of the form
	\[
		X = c \Id + C,
	\]
	where $c$ is a real number, the following formula holds true
	\begin{equation} \label{part3:eq:4}
		\discr(X) = \discr(C).
	\end{equation}
	   
	Finally, one has the following inequality valid for any $X \in M_d(\CC)$
	\begin{equation} \label{intro:nierownoscNaNormyMacierzy}
		\lVert X \rVert \leq \lVert X \rVert_2 \leq \lVert X \rVert_1,
	\end{equation}
	where $\| X \|_t$, is the $t$-norm of the matrix considered as the element of $\CC^{d^2}$.

\subsection{The total N-variation} \label{sec:prelim:variation}
	Given a positive integer $N$, we define the total $N$-variation $\calV_N$ of a sequence of vectors 
	$x = (x_n : n \geq m)$ from a vector space $V$ by
	\begin{equation} \label{eq:definicjaWahania}
		\calV_N(x) = \sum_{n = m}^\infty \norm{x_{n + N} - x_n}.
	\end{equation}
	Observe that if $(x_n : n \geq 0)$ has finite total $N$-variation then for each 
	$j \in \{0, \ldots, N-1\}$ a subsequence $(x_{k N + j} : k \geq 0 )$ is a Cauchy sequence. 

	The following Proposition collects some of the basic properties of total $N$-variation.
	The proof is straightforward.
	\begin{proposition} \label{prop:6}
		Let $V$ be a normed algebra, then the total $N$-variation has the following properties.
		\begin{enumerate}[(a)]
			\item $\begin{aligned}[b]
				\sup_{n \geq m} {\norm{x_n}} \leq \calV_N(x_n : n \geq m) + 
				\max\big\{\norm{x_m}, \ldots, \norm{x_{m+N-1}}\big\},
            \end{aligned}$ \label{prop:6a} 
            
            \item $\begin{aligned}[b]
               \calV_{N M} (x_n : n \geq m) \leq M \calV_N (x_n : n \geq m),
            \end{aligned}$ \label{prop:6b}
            
            \item $\begin{aligned}[b]
               \calV_{N} (x_n + y_n: n \geq m) \leq 
               \calV_{N} (x_n : n \geq m) + \calV_{N} (y_n : n \geq m),
            \end{aligned}$ \label{prop:6c}
            
            \item $\begin{aligned}[b]
               \calV_N(x_n y_n : n \geq m) \leq \sup_{n \geq m}{\norm{x_n}}\ \calV_N(y_n : n \geq m) +
               \sup_{n \geq m}{\norm{y_n}}\ \calV_N(x_n : n \geq m),
            \end{aligned}$ \label{prop:6d} 
            
            \item $\begin{aligned}[b]
               \calV_N(x_n^{-1} : n \geq m) \leq \sup_{n \geq m} \lVert x_n^{-1} \rVert^2 
               \calV_N(x_n : n \geq m).
            \end{aligned}$ \label{prop:6e}
		\end{enumerate}
	\end{proposition}
      
	In the sequel we use Banach algebra $V = C(K; M_2(\CC))$ of continuous functions on a compact set 
	$K \subset \CC$ with values in $M_2(\CC)$. The associated norm is as follows
	\begin{equation} \label{eq:5}
		\normM{f} = \sup_{x \in K} \norm{f(x)},
	\end{equation}
	where $\norm{\cdot}$ is the operator norm.

\subsection{Transfer matrices, generalised eigenvectors and associated orthonormal polynomials}
The sequence $u = (u_n : n \geq 0)$ is called generalised eigenvector associated with $x \in \RR$
when $u$ satisfies recurrence relation \eqref{eq:53}. In Introduction we explained how orthonormal
polynomials are related to $u$.

In the sequel we need the following notion. For $k \geq 0$ we define $k$th associated \emph{orthonormal}
polynomials by
\[
	\begin{gathered}
		p^{[k]}_0(x) = 1, \qquad p^{[k]}_1(x) = \frac{x - b_k}{a_k}, \\
		a_{n+k-1} p^{[k]}_{n-1}(x) + b_{n+k} p^{[k]}_n(x) + a_{n+k} p^{[k]}_{n+1}(x) = 
			x p^{[k]}_n(x), \quad (n \geq 1).
	\end{gathered}
\]
These are orthonormal polynomials associated with the sequences
\[
	a^{[k]} = (a_{n+k} : n \geq 0), \qquad b^{[k]} = (b_{n+k} : n \geq 0).
\]
For $k=0$ we usually omit the superscript.

With sequences $a = (a_n : n \geq 0)$ and $b = (b_n : n \geq 0)$ of positive and real numbers,
respectively, the associated \emph{transfer matrix} is defined by
\[
	B_n(x) =
	\begin{pmatrix}
		0 & 1 \\
		-\frac{a_{n-1}}{a_n} & \frac{x - b_n}{a_n}
	\end{pmatrix}, \quad (x \in \CC)
\]
for $n > 0$. If sequences $a$ and $b$ are in fact defined for $n \in \ZZ$, then we define $B_n$ also
for $n \leq 0$. If $u$ is generalised eigenvector associated with $x \in \RR$, then
\[
	\begin{pmatrix}
		u_n \\
		u_{n+1}
	\end{pmatrix} 
	=
	B_n(x) 
	\begin{pmatrix}
		u_{n-1} \\
		u_{n}
	\end{pmatrix} 	
\]
for all $n$ such that the right-hand side is well-defined.

In what follows the properties of the transfer matrices and the orthonormal polynomials associated to 
$N$-periodic sequences $a$ and $b$ will be crucial. For this reason we fix for them special
notation: we denote by $\hat{\calB}^i$ the corresponding transfer matrix $B_i$ and by $w^{[k]}$
the corresponding sequence of polynomials $p^{[k]}$. If $a$ and $b$ are $N$-periodic, we usually denote 
them by $\alpha$ and $\beta$, respectively. By $N$-periodicity, we will consider $\alpha$ and $\beta$ as
sequences defined on $\ZZ$.

\section{Algebraic identities} \label{sec:algIdent}
The aim of this section is to collect some of the identities which will be needed in various parts of 
this article.

The following Proposition is taken from \cite[Proposition 21.1]{Simon2008}. Note that formula (21.1) 
in that article has a typo: as it is easily seen for $n=1$, the proposed sum should be with the minus sign.
Below we provide the corrected formula.
\begin{proposition} \label{prop:9}
	Let $(w_n : n \geq 0)$ be a sequence of orthonormal polynomials associated with $\alpha$ and $\beta$.
	Then for every $x$ one has
	\[
		w_n'(x) = \frac{1}{\alpha_0} 
		\sum_{m=0}^{n-1} \left[ w_m(x) w_{n-1}^{[1]}(x) - w_n(x) w_{m-1}^{[1]}(x) \right] w_m(x).
	\]
\end{proposition}

The next Proposition is well-known but it is usually stated in a different form (see, e.g.,
\cite[formula (3.2.19)]{Simon2010}).
\begin{proposition} \label{prop:13}
	Let $(w_n : n \geq 0)$ be a sequence of orthonormal polynomials associated with $N$-periodic sequences 
	$\alpha$ and $\beta$. Then for every $n \geq 2$
	\[
		\prod_{j=0}^{n-1} \hat{\calB}^j(x) =
		\begin{pmatrix}
			-\frac{\alpha_{N-1}}{\alpha_0} w_{n-2}^{[1]}(x) & w_{n-1}(x) \\
			-\frac{\alpha_{N-1}}{\alpha_0} w_{n-1}^{[1]}(x) & w_{n}(x)
		\end{pmatrix}.
	\]
\end{proposition}
\begin{proof}
	Direct computation shows that the formula is valid for $n=2$. Therefore, simple induction together with
	recurrence relation for $w$ and $w^{[1]}$ completes the proof.
\end{proof}

The following Proposition is an extension of Turán identity for Chebyshev polynomials of the second
kind (obtained for $\alpha_n \equiv 1$ and $\beta_n \equiv 0$).
\begin{proposition} \label{prop:11}
	Let $(w_i : i \geq 0)$ be a sequence of orthonormal polynomials associated with $N$-periodic sequences 
	$\alpha$ and $\beta$. Then
	\[
		w_{i}(x) w_{i}^{[1]}(x) - w_{i+1}(x) w_{i-1}^{[1]}(x) = 
		\frac{\alpha_0}{\alpha_{i}}
		\quad (i \geq 1).
	\]
\end{proposition}
\begin{proof}
	By Proposition~\ref{prop:13} for every $i \geq 1$
	\[
		\prod_{j=0}^{i} \hat{\calB}^j(x) =
		\begin{pmatrix}
			-\frac{\alpha_{N-1}}{\alpha_0} w_{i-1}^{[1]}(x) & w_{i}(x) \\
			-\frac{\alpha_{N-1}}{\alpha_0} w_{i}^{[1]}(x) & w_{i+1}(x)
		\end{pmatrix}.
	\]
	Taking the determinant of both sides gives
	\[
		\frac{\alpha_{N-1}}{\alpha_i} = 
		\frac{\alpha_{N-1}}{\alpha_0} \left[ w_{i}(x) w_{i}^{[1]}(x) - w_{i+1}(x) w_{i-1}^{[1]}(x) \right]
	\]
	and the result follows.
\end{proof}

In the next result we provide a simple formula for the inverse of finite Jacobi matrices in terms of
the associated orthonormal polynomials. 
\begin{proposition} \label{prop:15}
	Let $(w^{[k]}_i(x) : i \in \mathbb{N})$ be a sequence of \emph{orthonormal} polynomials associated with
	sequences $\alpha^{[k]} = (\alpha_{i+k} : i \in \mathbb{N})$ and 
	$\beta^{[k]} = (\beta_{i+k} : i \in \mathbb{N})$. Then for
	\[
		D_M(x) = \tridiag(\alpha_i, \beta_i - x, \alpha_i : 0 \leq i < M)
	\]
	we have
	\[
		[D^{-1}_M(x)]_{i, j} = 
		\begin{cases}
			-\frac{w_{j-1}(x) w_{M-i}^{[i]}(x)}{\alpha_{i-1} w_M(x)} & M \geq i \geq j \geq 1, \\
			[D^{-1}_M(x)]_{j, i} & \text{otherwise.}
		\end{cases}
	\]
\end{proposition}
\begin{proof}
	It is well-known that for every $k \geq 0$ we have
	\[
		W_k(x) = \det( -D_k(x) ),
	\]
	where $(W_k : k \geq 0)$ is the corresponding sequence of \emph{monic} orthogonal polynomials.
	Hence, the formula for the inverse matrix in terms of the adjugate matrix implies for 
	$M \geq i \geq j \geq 1$
	\begin{multline*}
		[-D^{-1}_M(x)]_{i, j} = (-1)^{i+j} \frac{W_{j-1}(x) W^{[i]}_{M-i}(x)}{W_M(x)}
		[(-\alpha_{j-1})(-\alpha_{j-1}) \cdots (-\alpha_{i-2})] \\= 
		\frac{W_{j-1}(x) W^{[i]}_{M-i}(x)}{W_M(x)} (\alpha_{j-1} \cdots \alpha_{i-2}).
	\end{multline*}     
	Comparing coefficients at $x^k$ implies
	\[
		W_k(x) = (\alpha_{0} \cdots \alpha_{k-1}) w_k(x).
	\]
	Hence,
	\begin{multline*}
		[-D^{-1}_M(x)]_{i, j} = 
		\frac{(\alpha_0 \cdots \alpha_{j-2}) w_{j-1}(x) (\alpha_i \cdots \alpha_{M-1}) w^{[i]}_{M-i}(x)}
		{(\alpha_0 \cdots \alpha_{M-1}) w_M(x)} 
		(\alpha_{j-1} \cdots \alpha_{i-2}) \\= 
		\frac{w_{j-1}(x) w_{M-i}^{[i]}(x)}{\alpha_{i-1} w_M(x)}.
	\end{multline*}
	For $M \geq j \geq i \geq 1$ the result follows from the fact that the inverse of a symmetric matrix
	is also symmetric.
\end{proof}

The following technical Proposition provides the link to geometric interpretation of the critical case.
\begin{proposition} \label{prop:10}
	Assume that for some $x_0 \in \RR$ and $\gamma \in \{ -1, 1 \}$
	\[
		\prod_{i=0}^{N-1} \hat{\calB}^i(x_0) = \gamma \Id.
	\]
	Then
	\[
		(w_N)'(x_0) = \frac{\alpha_{N-1}}{\alpha_0} (w^{[1]}_{N-2})'(x_0), \quad
		w_{N-1}'(x_0) (w_{N-1}^{[1]})'(x_0) > w_N'(x_0) (w_{N-2}^{[1]})'(x_0).
	\]
\end{proposition}
\begin{proof}	
	First of all, 
	\[
		\prod_{i=0}^{N-1} \hat{\calB}^{i}(x_0) = \gamma \Id \quad \Rightarrow  \quad
		\prod_{i=k}^{k+N-1} \hat{\calB}^{i}(x_0) = \gamma \Id \quad (k \geq 0).
	\]
	Indeed, it follows from the fact that these matrices are conjugated to each other. Hence,
	\begin{equation} \label{eq:11}
		-\frac{\alpha^{[k]}_{N-1}}{\alpha^{[k]}_0} w_{N-2}^{[k+1]}(x_0) = w_N^{[k]}(x_0) = \gamma, \quad 
		w_{N-1}^{[k]}(x_0) = w_{N-1}^{[k+1]}(x_0) = 0, 
		\quad (k \geq 0).
	\end{equation}
	
	Proposition~\ref{prop:11} for $i=N-1$ and for sequences $\alpha^{[k]}$ and $\beta^{[k]}$ gives
	\[
		 w^{[k]}_{N-1}(x) w^{[k+1]}_{N-1}(x) - w^{[k]}_N(x) w^{[k+1]}_{N-2}(x) = 
		 \frac{\alpha^{[k]}_0}{\alpha^{[k]}_{N-1}}, \quad (k \geq 0).
	\]
	Taking derivatives at $x=x_0$ and applying \eqref{eq:11} gives
	\[ 
		0 = (w^{[k]}_N)'(x_0) w^{[k+1]}_{N-2}(x_0) + w^{[k]}_N(x_0) (w^{[k+1]}_{N-2})'(x_0) = 
		w^{[k]}_N(x_0) 
		\left[ -\frac{\alpha^{[k]}_0}{\alpha^{[k]}_{N-1}} (w^{[k]}_N)'(x_0) + 
		(w^{[k+1]}_{N-2})'(x_0) \right].
	\]
	Hence,
	\begin{equation} \label{eq:12}
		(w^{[k]}_N)'(x_0) = \frac{\alpha^{[k]}_{N-1}}{\alpha^{[k]}_0} (w^{[k+1]}_{N-2})'(x_0) \quad 
		(k \geq 0)
	\end{equation}
	and by taking $k=0$ we obtain the first desired identity.
	
	According to Proposition~\ref{prop:9}
	\[
		w_{N-1}'(x_0) = \frac{w_{N-2}^{[1]}(0)}{\alpha_{0}} \sum_{m=0}^{N-2} |w_m(x_0)|^2, \quad
		(w_{N-1}^{[1]})'(x_0) = \frac{w_{N-2}^{[2]}(0)}{\alpha_{1}} \sum_{m=0}^{N-2} |w^{[1]}_m(x_0)|^2.
	\]
	By $\gamma^2=1$ and \eqref{eq:11} for $k \in \{0,1\}$, we obtain
	\begin{equation} \label{eq:9}
		\frac{\alpha_{N-1}}{\alpha_0} w_{N-1}'(x_0) (w_{N-1}^{[1]})'(x_0) = 
		\frac{1}{\alpha_0^2} \sum_{m=0}^{N-2} |w_m(x_0)|^2 \sum_{m=0}^{N-2} |w^{[1]}_m(x_0)|^2.
	\end{equation}
	
	Proposition~\ref{prop:9} and \eqref{eq:11} implies
	\[
		w_N'(x_0) = \frac{1}{\alpha_0} \sum_{m=0}^{N-1} w_{m-1}^{[1]}(x_0) w_m(x_0) = 
		\frac{1}{\alpha_0} \sum_{m=1}^{N-2} w_{m-1}^{[1]}(x_0) w_m(x_0).
	\]
	By Schwarz inequality
	\[
		|w_N'(x_0)|^2 \leq \frac{1}{\alpha_0^2} \sum_{m=0}^{N-3} (w_{m}^{[1]}(x_0))^2 
		\sum_{m=1}^{N-2} (w_{m}(x_0))^2.
	\]
	Hence, using $w_{0}(x_0) = 1 > 0$, \eqref{eq:11} for $k=0$ and \eqref{eq:9} we get
	\[
		\frac{\alpha_{N-1}}{\alpha_0} w_{N-1}'(x_0) (w_{N-1}^{[1]})'(x_0) - |w_N'(x_0)|^2 > 0
	\]
	and by \eqref{eq:12} for $k=0$ the result follows.
\end{proof}

\section{Self-adjointness} \label{sec:selfAdj}

Theorem~\ref{thm:A} is a consequence of the following result.
\begin{theorem}
	Let $\alpha = (\alpha_n : n \in \mathbb{Z})$ and $\beta = (\beta_n : n \in \mathbb{Z})$ be
	real $N$-periodic sequences ($N \geq 2,\ \alpha_n > 0$). Assume
	\[
		\calF(0) := \prod_{i=0}^{N-1} 
		\begin{pmatrix}
			0 & 1 \\
			-\frac{\alpha_{i-1}}{\alpha_i} & -\frac{\beta_i}{\alpha_i}
		\end{pmatrix} = \gamma \Id.
	\]
	Define
	\[
		a_{Nk+i} = \alpha_i \tilde{a}_k, \quad 
		b_{Nk+i} = \beta_i \tilde{a}_k \quad 
		(i = 0, 1, \ldots, N-1,\ k \geq 0)
	\]
	for a sequence $( \tilde{a}_k : k \geq 0)$. Then $|\gamma| = 1$ and
	\[
		p_{Nk+i}(0) = \gamma^k w_i(0),
	\]
	where $(w_n : n \in \mathbb{N})$ is the sequence of orthonormal polynomials associated with $\alpha$ and 
	$\beta$. In particular, the operator $A$ is \emph{always} self-adjoint and $0 \notin \sigmaP{A}$.
\end{theorem}
\begin{proof}
	For simplicity let us denote $p_n \equiv p_n(0)$. Let us show that $p_{Nk-1} = 0$ and 
	$p_{Nk} = \gamma^k$. 
	For $k=0$ it holds. For $k \geq 0$ by the recurrence relation
	\[
		\begin{pmatrix}
			p_{N(k+1)-1} \\
			p_{N(k+1)}
		\end{pmatrix}
		=
		\calF(0)
		\begin{pmatrix}
			0 & 1 \\
			-\frac{\alpha_{N-1}}{\alpha_0} & -\frac{\beta_0}{\alpha_0}
		\end{pmatrix}^{-1}
		\begin{pmatrix}
			0 & 1 \\
			-\frac{\alpha_{N-1}}{\alpha_0} \frac{\tilde{a}_{k-1}}{\tilde{a}_k} & -\frac{\beta_0}{\alpha_0}
		\end{pmatrix}
		\begin{pmatrix}
			p_{Nk - 1} \\
			p_{Nk}
		\end{pmatrix}.
	\]
	By the induction hypothesis and the fact that $\calF(0) = \gamma \Id$ it is equal to
	\begin{multline*}
		\gamma^{k+1}
		\begin{pmatrix}
			0 & 1 \\
			-\frac{\alpha_{N-1}}{\alpha_0} & -\frac{\beta_0}{\alpha_0}
		\end{pmatrix}^{-1}
		\begin{pmatrix}
			0 & 1 \\
			-\frac{\alpha_{N-1}}{\alpha_0} \frac{\tilde{a}_{k-1}}{\tilde{a}_k} & -\frac{\beta_0}{\alpha_0}
		\end{pmatrix}
		\begin{pmatrix}
			0 \\
			1
		\end{pmatrix} \\ = 
		\gamma^{k+1} \frac{\alpha_0}{\alpha_{N-1}}
		\begin{pmatrix}
			-\frac{\beta_0}{\alpha_0} & -1 \\
			\frac{\alpha_{N-1}}{\alpha_0} & 0
		\end{pmatrix}
		\begin{pmatrix}
			1 \\
			-\frac{\beta_0}{\alpha_0}
		\end{pmatrix}
		= \gamma^{k+1}
		\begin{pmatrix}
			0 \\
			1
		\end{pmatrix}.
	\end{multline*}
	What was to be shown.
	
	Assume now that $0 < j < N$. By the recurrence relation
	\[
		\begin{pmatrix}
			p_{Nk+j-1} \\
			p_{Nk+j} 
		\end{pmatrix} =
		\prod_{i=1}^{j-1} B_i(0) \cdot
		\begin{pmatrix}
			0 & 1 \\
			-\frac{\alpha_{N-1}}{\alpha_0} \frac{\tilde{a}_{k-1}}{\tilde{a}_k} & -\frac{\beta_0}{\alpha_0}
		\end{pmatrix}
		\begin{pmatrix}
			p_{Nk-1} \\
			p_{Nk}
		\end{pmatrix}.
	\]
	The previous case implies
	\[
		\begin{pmatrix}
			p_{Nk+j-1} \\
			p_{Nk+j} 
		\end{pmatrix} = 
		\gamma^k \prod_{i=1}^{j-1} B_i(0) \cdot
		\begin{pmatrix}
			0 & 1 \\
			-\frac{\alpha_{N-1}}{\alpha_0} \frac{\tilde{a}_{k-1}}{\tilde{a}_k} & -\frac{\beta_0}{\alpha_0}
		\end{pmatrix}
		\begin{pmatrix}
			0 \\ 
			1
		\end{pmatrix}
		= \gamma^k \prod_{i=1}^{j-1} B_i(0) \cdot 
		\begin{pmatrix}
			p_0 \\
			p_1
		\end{pmatrix}
		= \gamma^k
		\begin{pmatrix}
			p_{j-1} \\ 
			p_j
		\end{pmatrix}.
	\]
	The equality $p_j(0) = w_j(0)$ for $0 \leq j < N$ follows by $p_0(0) = w_0(0), p_1(0) = w_1(0)$ and 
	the fact that both sequences satisfy the same recurrence relation.
	         
	Let us notice that from the periodicity of $\alpha$
	\[
		1 = \det{\calF} = \gamma^2.
	\]
	Hence $|\gamma|=1$, and consequently the sequence $(p_n(0) : n \geq 0)$ is $2N$-periodic, so
	it cannot belong to $\ell^2$. Therefore, according to \cite[Theorem 3]{Simon1998} the operator $A$ is 
	self-adjoint and $0 \notin \sigmaP{A}$.
\end{proof}
       
\section{The formula for the resolvent} \label{sec:formRes}
\subsection{Definitions}
Let us define matrices
\begin{align*}
	D_M(x) &= \tridiag \left(\alpha_{i}, \beta_{i} - x, \alpha_{i} : 0 \leq i < M \right), \\
	E_{i,j} &= \delta_1(j) \delta_N(i).
\end{align*}

Jacobi matrix $\calA$ associated with the sequences defined in \eqref{eq:58} can be equivalently expressed
in the following block form
\begin{equation*}
	\calA =
	\begin{pmatrix}
		\tilde{a}_0 D_N(0) & \alpha_{N-1} \tilde{a}_0 E & 0 & 0 &\ldots \\
		\alpha_{N-1} \tilde{a}_0 E^* & \tilde{a}_1 D_N(0) & \alpha_{N-1} \tilde{a}_1 E & 0 & \ldots \\
		0   & \alpha_{N-1} \tilde{a}_1 E^* & \tilde{a}_2 D_N(0) & \alpha_{N-1} \tilde{a}_2 E & \ldots \\
		0   & 0 & \alpha_{N-1} \tilde{a}_2 E^* & \tilde{a}_3 D_N(0)   &  \\
		\vdots & \vdots & \vdots  &  & \ddots
	\end{pmatrix}.	
\end{equation*}

Let us define matrix $\calB$ by
\[
	\calB_{i,j} = 
	\begin{cases}
		\frac{\gamma^{j-i}}{\tilde{a}_j} F & \text{for } j > i, \\
		\frac{1}{\tilde{a}_i} D_N^{-1}(0) & \text{for } j=i, \\
		[\calB_{j,i}]^* & \text{otherwise}.
	\end{cases},
\]
where
\[
	\gamma = w_N(0)
\]
and
\[
	F_{i,j} = \frac{1}{\gamma} [D_{2N}^{-1}(0)]_{i, N+j}, \quad (i,j \in \{1, 2, \ldots, N\}).
\]
Proposition~\ref{prop:15} implies the formula
\begin{equation} \label{eq:21}
	F_{i,j} = -\frac{w_{i-1}(0) w_{N-j}^{[j]}(0)}{\gamma \alpha_{j-1}}.
\end{equation}

\subsection{Properties}
The aim of this section is to prove the following result.
\begin{theorem} \label{thm:1}
	Let the assumptions of Theorem~\ref{thm:A} be satisfied. Then we have the equality of matrices
	\[
		\calA \calB = \calB \calA = \Id.
	\]
\end{theorem}

We need the following algebraic lemma.
\begin{lemma} \label{lem:1}
	We have the following identities.
	\begin{align}
		F E^* &= 0, \label{lem:1:eq:1} \\
		E^* F &= 0, \label{lem:1:eq:2} \\
		D_N^{-1}(0) E - F E &= 0, \label{lem:1:eq:4} \\
		\gamma F D_N(0) + \alpha_{N-1} F E &= 0. \label{lem:1:eq:3}
	\end{align}
\end{lemma}
\begin{proof}
	We have
	\[
		(F E^*)_{i,j} = \sum_{k=1}^N F_{i,k} E^*_{k,j} = 
		\sum_{k=1}^N F_{i,k} \delta_1(k) \delta_N(j) = F_{i,1} \delta_N(j) = 
		-\frac{w_{i-1}(0) w_{N-1}^{[1]}(0)}{\gamma \alpha_{0}} \delta_N(j) = 0
	\]	
	and \eqref{lem:1:eq:1} follows. Similarly,
	\[
		(E^* F)_{i,j} = \sum_{k=1}^N E^*_{i,k} F_{k,j} = 
		\sum_{k=1}^N \delta_1(i) \delta_N(k) F_{k,j} = \delta_1(i) F_{N, j} =
		-\frac{w_{N-1}(0) w^{[j]}_{N-j}(0)}{\gamma \alpha_{j-1}} \delta_1(i) = 0,
	\]
	which shows \eqref{lem:1:eq:2}
	
	One has
	\[
		(F E)_{i,j} = \sum_{k=1}^N F_{i,k} E_{k,j} = 
		\sum_{k=1}^N F_{i,k} \delta_1(j) \delta_N(k) = \delta_1(j) F_{i,N} = 
		- \delta_1(j) \frac{w_{i-1}(0)}{\gamma \alpha_{N-1}}.
	\]
	Similarly,
	\[
		(D_N^{-1}(0) E)_{i,j} = \delta_1(j) [D_N^{-1}(0)]_{i,N} = 
		- \delta_1(j) \frac{w_{i-1}(0)}{\gamma \alpha_{N-1}}.
	\]
	Hence, we have shown \eqref{lem:1:eq:4}.
	
	Observe
	\[
		D_{2N}^{-1}(0) D_{2N}(0) =
		\begin{pmatrix}
			G_1 & \gamma F \\
			\gamma F^* & G_2
		\end{pmatrix}
		\begin{pmatrix}
			D_N(0) & \alpha_{N-1} E \\
			\alpha_{N-1} E^* & D_N(0)
		\end{pmatrix}
		= 
		\begin{pmatrix}
			\Id & 0 \\
			0 & \Id
		\end{pmatrix}
	\]	
	for some matrices $G_1, G_2 \in M_N(\RR)$. It implies
	\[
		G_1 D_N(0) + \gamma \alpha_{N-1} F E^* = \Id.
	\]
	Since $F E^* = 0$, we obtain $G_1 = D_N^{-1}(0)$. Consequently,
	\[
		\alpha_{N-1} D_N^{-1}(0) E + \gamma F D_N(0) = 0
	\]
	which by \eqref{lem:1:eq:4} implies \eqref{lem:1:eq:3}. The proof is complete.
\end{proof}

We are ready to prove Theorem~\ref{thm:1}.
\begin{proof}[Proof of Theorem~\ref{thm:1}]
	By the symmetry of matrices $\calA$ and $\calB$ it is enough to show $\calB \calA = \Id$. We have
	\begin{multline*}
		(\calB \calA)_{i,j} = \sum_{k=0}^\infty \calB_{i,k} \calA_{k,j} = 
		\calB_{i, j-1} \calA_{j-1,j} + \calB_{i, j} \calA_{j,j} + \calB_{i, j+1} \calA_{j+1,j} \\=
		\calB_{i, j-1} \alpha_{N-1} \tilde{a}_{j-1} E + \calB_{i, j} \tilde{a}_j D_N(0) + 
		\calB_{i, j+1} \alpha_{N-1} \tilde{a}_j E^*.
	\end{multline*}
	
	Again, by the symmetry it is enough to prove the thesis for $j \geq i$.
	
	Consider the case $i=j$. Then
	\[
		(\calB \calA)_{i,i} = \frac{\gamma}{\tilde{a}_i} F^* \alpha_{N-1} \tilde{a}_{i-1} E +
		\frac{1}{\tilde{a}_i} D_N^{-1}(0) \tilde{a}_i D_N(0) + 
		\frac{\gamma}{\tilde{a}_{i+1}} F \alpha_{N-1} \tilde{a}_i E^*,
	\]
	what from the formulas \eqref{lem:1:eq:1} and \eqref{lem:1:eq:2} is equal to $\Id$.
	
	Assume now that $j=i+1$. Then
	\begin{multline*}
		(\calB \calA)_{i,i+1} = \frac{1}{\tilde{a}_i} D^{-1}_N(0) \alpha_{N-1} \tilde{a}_i E + 
		\frac{\gamma}{\tilde{a}_{i+1}} F \tilde{a}_{i+1} D_N(0) + 
		\frac{\gamma^2}{\tilde{a}_{i+2}} F \alpha_{N-1}\tilde{a}_{i+1} E^* \\= 
		\alpha_{N-1} D^{-1}_N(0) E + \gamma F D_N(0),
	\end{multline*}
	what from the formulas \eqref{lem:1:eq:3} and \eqref{lem:1:eq:4} is equal to $0$.
	
	Finally, for $j > i+1$ we have
	\begin{multline*}
		(\calB \calA)_{i, j} = \frac{\gamma^{j-i-1}}{\tilde{a}_{j-1}} F \alpha_{N-1} \tilde{a}_{j-1} E + 
		\frac{\gamma^{j-i}}{\tilde{a}_j} F \tilde{a}_j D_N(0) + 
		\frac{\gamma^{j-i+1}}{\tilde{a}_{j+1}} F \alpha_{N-1}\tilde{a}_j E^* \\= 
		\gamma^{j-i-1} (\alpha_{N-1} F E + \gamma F D_N(0))
	\end{multline*}
	what from the formulas \eqref{lem:1:eq:3} and \eqref{lem:1:eq:4} is equal to $0$. The proof is complete.
\end{proof}

\section{Spectral phase transition} \label{sec:spectPhaseTrans}
We define operator $B$ as the restriction of $\calB$ to $\ell^2(\NN; \CC^N)$, i.e. 
\[
	B x = \calB x, \quad x \in \Dom(B) = \{ y \in \ell^2(\NN; \CC^N) : \calB y \in \ell^2(\NN; \CC^N) \},
\]
where
\[
	\sprod{x}{y}_{\ell^2(\NN; \CC^N)} = \sum_{n=0}^\infty \sprod{x_n}{y_n}_{\CC^N}, \quad 
	\ell^2(\NN; \CC^N) = \{ x \in (\CC^N)^\NN \colon \sprod{x}{x}_{\ell^2(\NN; \CC^N)} < \infty \}.
\]

Let us define $U : \ell^2(\NN) \rightarrow \ell^2(\NN; \CC^N)$ by the formula
\[
	(U x)_k = (x_{kN}, x_{kN+1}, \ldots, x_{kN+N-1}).
\]
It is clear that $U$ is a unitary operator.

Theorem~\ref{thm:1} implies that $A^{-1}$ exists if and only if operator $B$ is bounded. In such a case
we have
\[
	U B U^{-1} = A^{-1}.
\]
Hence, Theorem~\ref{thm:B} is a consequence of the following result. Its proof is inspired by
\cite[Theorem 2.1 and 2.2]{Sahbani2008} and \cite[Theorem 2.3]{DombrowskiPedersen1995}.
\begin{theorem} \label{thm:2}
	Let the assumptions of Theorem~\ref{thm:A} be satisfied.
	If
	\begin{enumerate}[(a)]
		\item $\begin{aligned}[b]
			\lim_{k \rightarrow \infty} \frac{k}{\tilde{a}_k} = \infty
		\end{aligned}$, 
		then $B$ is unbounded;
		
		\item the sequence 
		$\begin{aligned}[b]
			\left( \frac{k}{\tilde{a}_k} : k \geq 1 \right)
		\end{aligned}$
		is bounded, then $B$ is bounded;
		
		\item $\begin{aligned}[b]
			\lim_{k \rightarrow \infty} \frac{k}{\tilde{a}_k} = 0
		\end{aligned}$,
		then $B$ is compact.
	\end{enumerate}
\end{theorem}
\begin{proof}
	We have to show that $\calB$ restricted to $\ell^2(\NN; \CC^N)$ has the requested properties.
	
	Decompose operator $\calB$ by
	\begin{equation} \label{eq:16}
		\calB = \calL + \calX + \calL^*,
	\end{equation}
	where $\calX$ is a block diagonal and $\calL$ is a block lower triangular matrix.
	For any $x \in \ell^2(\NN; \CC^N)$ we have
	\begin{equation} \label{eq:14}
		(\calX x)_k = \frac{1}{\tilde{a}_k} D_N^{-1}(0) x_k
	\end{equation}
	and for $k>0$
	\begin{equation} \label{eq:15}
		(\calL x)_k = \frac{1}{\tilde{a}_k} \sum_{j=0}^{k-1} \gamma^{k-j} F^* x_j = 
		\frac{k}{\tilde{a}_k} \frac{1}{k} \sum_{j=0}^{k-1} \gamma^{k-j} F^* x_j.
	\end{equation}

	Assume that the sequence $(k / \tilde{a}_{k} : k \geq 1)$ is bounded. Then $\tilde{a}_{k}$
	is divergent to $\infty$. Hence, by \eqref{eq:14} the operator $\calX$ is compact as a product of
	a compact and a bounded operator.
	Observe
	\[
		\bigg\| \frac{1}{k} \sum_{j=0}^{k-1} \gamma^{k-j} F^* x_j \bigg\|_{\CC^N} \leq
		\norm{F} \frac{1}{k} \sum_{j=0}^{k-1} \norm{x_j}_{\CC^N}.
	\]
	Hence, by Hardy's inequality (see \cite{Hardy1920} for a original proof with a slightly worse constant)
	\[
		\sum_{k=1}^\infty \bigg\| \frac{1}{k} \sum_{j=0}^{k-1} \gamma^{k-j} F^* x_j \bigg\|_{\CC^N}^2 \leq
		4 \norm{F}^2 \norm{x}_{\ell^2(\NN; \CC^N)}^2.
	\]
	Formula \eqref{eq:15} implies that $\calL$ (and consequently also $\calL^*$) is bounded. Therefore, by
	\eqref{eq:16} operator $\calB$ is bounded. Moreover, if
	\[
		\lim_{k \rightarrow \infty} \frac{k}{\tilde{a}_{k}} = 0,
	\]
	then $\calL$ is a product of a compact and a bounded operator. Hence, $\calL$ is compact.
	It implies that $\calL^*$ is compact as well. Therefore, $\calB$ is compact as a sum of compact operators.
	
	Assume now that
	\begin{equation} \label{eq:23}
		\lim_{k \rightarrow \infty} \frac{k}{\tilde{a}_{k}} = \infty.
	\end{equation}
	For every $M \geq 1$ define\footnote{For $x \in \RR$ we define $x^+ = \max(0, x)$.}
	\[
		x^M_n = (-\gamma^n (M-n)^+, 0, 0,  \ldots, 0) \in \CC^N.
	\]
	By \eqref{eq:15}
	\[
		[(\calL x^M)_k]_N = -\frac{\gamma^{k}}{\tilde{a}_k} \sum_{j=0}^{k-1} (M-j)^+ [F^*]_{N,1}.
	\]
	By the fact that the first column of $F$ is zero, we have
	\[
		\calL^* x^M = 0.
	\]
	Proposition~\ref{prop:15} and formula \eqref{eq:21} imply
	\[
		[F^*]_{N,1} = -\frac{\gamma}{\alpha_{N-1}} = [D_N^{-1}(0)]_{N,1},
	\]
	and consequently,
	\[
		[(\calB x^M)_k]_N = \frac{\gamma^{k+1}}{\alpha_{N-1} \tilde{a}_k} \sum_{j=0}^{k} (M-j)^+.
	\]
	Hence,
	\begin{multline} \label{eq:24}
		\frac{1}{\norm{x^M}^2} \norm{\calB x^M}^2 \geq
		\frac{1}{\norm{x^M}^2} \sum_{k=0}^{M-1} |[(\calB x^M)_k]_N|^2 =
		\frac{1}{\alpha_{N-1}^2 \norm{x^M}^2} 
		\sum_{k=0}^{M-1} \frac{1}{\tilde{a}_k^2} \left( \frac{M(M-k)}{2}\right)^2 \\
		\geq \frac{\norm{x^M}^2}{4 \alpha_{N-1}^2\norm{x^M}^2 }
		\frac{M^2}{\max_{0 \leq k \leq M} \tilde{a}_k^2} =
		\frac{1}{4 \alpha_{N-1}^2}
		\frac{M^2}{\max_{0 \leq k \leq M} \tilde{a}_k^2}.
	\end{multline}
	For every $M$ let $0 \leq k_M \leq M$ be such that
	\[
		\max_{0 \leq k \leq M} \tilde{a}_k = \tilde{a}_{k_M}.
	\]
	Then
	\[
		\frac{M}{\max_{0 \leq k \leq M} \tilde{a}_k} = \frac{M}{\tilde{a}_{k_M}}.
	\]
	Now we have two cases. If the sequence $(\tilde{a}_M : M \geq 0)$ attains its supremum,
	then the denominator of the right-hand side is eventually constant, and consequently, the whole
	expression tends to infinity. Otherwise, the sequence $(\tilde{a}_M : M \geq 0)$ does not attain 
	its supremum. Hence, the sequence $(k_M : M \geq 0)$ tends to infinity. Since
	\[
		\frac{M}{\max_{0 \leq k \leq M} \tilde{a}_k} \geq \frac{k_M}{\tilde{a}_{k_M}}
	\]
	and \eqref{eq:23} the right-hand side of the last inequality is unbounded as 
	$M$ tends to infinity. Therefore, by \eqref{eq:24} the operator $\calB$ cannot be bounded on 
	$\ell^2(\NN;\CC^N)$. The proof is complete.
\end{proof}

\section{Asymptotics of generalised eigenvectors} \label{sec:asympGenEig}
\subsection{Definitions}
For each $\lambda \in \RR$ and $n \in \NN$ we define a binary quadratic form on $\RR^2$ by
\begin{equation} \label{eq:45}
	Q_n^\lambda(v) = a_{n+N-1}\sprod{E X_n(\lambda) v}{v},
\end{equation}
where $X_n(\lambda)$ is given by
\begin{equation} \label{eq:32}
	X_n(\lambda) = \prod_{k=n}^{n+N-1} B_{k}(\lambda), \quad \text{for} \quad
	B_k(\lambda) =
	\begin{pmatrix}
		0 & 1 \\
		-\frac{a_{k-1}}{a_k} & \frac{\lambda - b_k}{a_k}
	\end{pmatrix}
\end{equation}
and $E$ by
\[
	E = 
	\begin{pmatrix}
		0 & -1 \\
		1 & 0
	\end{pmatrix}.
\]
Finally, we define $N$-shifted Turán determinants by
\begin{equation} \label{eq:17}
	S_n(\alpha, \lambda) = a_{n + N - 1} Q^\lambda_n
	\begin{pmatrix}
			u_{n-1} \\
            u_n
	\end{pmatrix} =
	a_{n+N-1}^2 [u_{n+N-1} u_n - u_{n+N} u_{n-1}],
\end{equation}
where $u$ is the generalised eigenvector corresponding to $\lambda$ such that $(u_0, u_1) = \alpha \in \RR^2 \setminus \{ 0 \}$.

The following Proposition suggests that sequences satisfying assumptions of Theorem~\ref{thm:C} 
are in fact perturbations of sequences considered in Theorem~\ref{thm:A} and \ref{thm:B}.
\begin{proposition} \label{prop:4}
	Assume we are given $N$-periodic sequences $r$ and $q$ such that $r_0 r_1 \ldots r_{N-1} = 1$.
	Then
	\[
		r_i = \frac{\alpha_{i-1}}{\alpha_i}, \quad q_i = \frac{\beta_i}{\alpha_i}
	\]
	for $N$-periodic sequences 
	\[
		\alpha_i = \prod_{k=i+1}^{N-1} r_k, \quad \beta_i = \alpha_i q_i \qquad (i=0,1,\ldots,N-1).
	\]
\end{proposition}
\begin{proof}
	Direct computation.
\end{proof}
For the rest of this article if not otherwise stated by $\alpha$ and $\beta$ we always mean the sequences 
from Proposition~\ref{prop:4}.

Following \cite{Swiderski2017}, we are going to show that for every $i$ the sequence 
$(S_{kN+i} : k \geq 0)$ is uniformly convergent on compact subsets of some open set $\Lambda \subset \RR$
to a continuous function which has no zeros on $\Lambda$. In light of Lemma~\ref{lem:2} below and
the techniques presented in \cite{Swiderski2017} we are going to prove an estimate of the kind
\[
	\sum_{n=0}^{\infty} \sup_{\lambda \in K} \norm{C_{n+N}(\lambda) - C_n(\lambda)} < \infty
\]
for every compact $K \subset \CC$ and
\[
	C_n(\lambda) = a_{n+N-1} (X_n(\lambda) - \gamma \Id).
\]
Since the proof of this estimate is rather complicated and technical, we divide the analysis to several
subsections.

\subsection{Technical estimates}
We start by considering a simplified version of our problem.
\begin{proposition} \label{prop:1}
	Let $N$ be a positive integer and $i \in \NN$. Define
	\begin{equation} \label{prop:1:eq:3}
		\hat{\calX}^i(\lambda) = \prod_{j=i}^{N+i-1} \hat{\calB}^j(\lambda), \quad \text{for} \quad
		\hat{\calB}^j(\lambda) = 
		\begin{pmatrix}
			0 & 1 \\
			-\frac{\alpha_{j-1}}{\alpha_j} & \frac{\lambda - \beta_j}{\alpha_j}
		\end{pmatrix}
	\end{equation}
	for $N$-periodic sequences $\alpha$ and $\beta$. Assume that $\hat{\calX}^i(0) = \gamma \Id$. Let
	\begin{equation} \label{prop:1:eq:2}
		\hat{C}^i_k(\lambda) = 
		\hat{a}_k \left[ \hat{\calX}^i \left( \frac{\lambda}{\hat{a}_k} \right) - \gamma \Id \right]
	\end{equation}
	for a sequence $\hat{a}$ such that
	\begin{equation} \label{prop:1:eq:4}
		\lim_{k \rightarrow \infty} \hat{a}_k = \infty.
	\end{equation}
	Fix a compact set $K \subset \CC$. Then
	\begin{equation} \label{prop:1:eq:1}
		\hat{\calC}^i(\lambda) := \lim_{k \rightarrow \infty} \hat{C}^i_k(\lambda) = \lambda
		\begin{pmatrix}
			-\frac{\alpha_{i-1}}{\alpha_i} (w_{N-2}^{[i+1]})'(0) & (w_{N-1}^{[i]})'(0) \\
			-\frac{\alpha_{i-1}}{\alpha_i} (w_{N-1}^{[i+1]})'(0) & (w_{N}^{[i]})'(0)
		\end{pmatrix}
	\end{equation}
	uniformly on $K$, where for every $k$ the symbol $w^{[k]}$ denotes 
	the sequence of orthonormal polynomials associated with $\alpha^{[k]}$ and $\beta^{[k]}$.
	Moreover, there is a constant $c > 0$ such that
	\begin{equation} \label{prop:1:eq:5}
		\sup_{\lambda \in K} \norm{\hat{C}^i_{k+1}(\lambda) - \hat{C}^i_k(\lambda)} \leq 
		c \left| \frac{1}{\hat{a}_{k+1}} - \frac{1}{\hat{a}_k} \right|.
	\end{equation}
\end{proposition}
\begin{proof}
	Fix a compact set $K \subset \CC$. First, by Proposition~\ref{prop:13}, we get
	\[
		\hat{\calX}^i(\lambda_k) =
		\begin{pmatrix}
			-\frac{\alpha_{i-1}}{\alpha_i} w_{N-2}^{[i+1]}(\lambda_k) & w_{N-1}^{[i]}(\lambda_k) \\
			-\frac{\alpha_{i-1}}{\alpha_i} w_{N-1}^{[i+1]}(\lambda_k) & w_N^{[i]}(\lambda_k)
		\end{pmatrix},
	\]
	where $\lambda_k = \lambda/\hat{a}_k$. By $\hat{\calX}^i(0) = \gamma \Id$ we have
	\[
		-\frac{\alpha_{i-1}}{\alpha_i} w_{N-2}^{[i+1]}(0) = w_N^{[i]}(0) = \gamma, \qquad 
		w_{N-1}^{[i]}(0) = w_{N-1}^{[i+1]}(0) = 0.
	\]
	Therefore, by \eqref{prop:1:eq:2}
	\begin{equation} \label{eq:19}
		\hat{C}^i_k(\lambda) = \hat{a}_k
		\begin{pmatrix}
			-\frac{\alpha_{i-1}}{\alpha_i} [w_{N-2}^{[i+1]}(\lambda_k) - w_{N-2}^{[i+1]}(0)] & 
				w_{N-1}^{[i]}(\lambda_k) \\
			-\frac{\alpha_{i-1}}{\alpha_i} w_{N-1}^{[i+1]}(\lambda_k) & 
				[w_N^{[i]}(\lambda_k) - w_N^{[i]}(0)]
		\end{pmatrix}.
	\end{equation}
	For every polynomial $p$ of degree $n$ and such that $p(0)=0$ Taylor expansion implies
	\begin{equation} \label{eq:20}
		\hat{a}_k p(\lambda_k) = -\lambda p'(0) - 
		\lambda \sum_{m=2}^n \frac{p^{(m)}(0)}{m!} \left( \frac{-\lambda}{\hat{a}_k} \right)^{m-1}.
	\end{equation}
	Assumption \eqref{prop:1:eq:4} implies $\lambda_k \rightarrow 0$. Hence, \eqref{eq:20} applied 
	to every entry in \eqref{eq:19} implies \eqref{prop:1:eq:1}. Moreover, \eqref{eq:20} implies
	\begin{multline*}
		\hat{a}_{k+1} p(\lambda_{k+1})  - \hat{a}_k p(\lambda_k) =
		-\lambda  \sum_{m=2}^n \frac{p^{(m)}(0)}{m!} (-\lambda)^{m} 
		\left( \frac{1}{(\hat{a}_{k+1})^m} - \frac{1}{(\hat{a}_{k})^m} \right) = \\
		-\lambda \left( \frac{1}{\hat{a}_{k+1}} - \frac{1}{\hat{a}_{k}} \right) 
		\sum_{m=2}^n \frac{p^{(m)}(0)}{m!} (-\lambda)^{m} 
		\sum_{r=0}^{m-1} \left( \frac{1}{\hat{a}_{k}} \right)^r 
		\left( \frac{1}{\hat{a}_{k+1}} \right)^{m-1-r}.
	\end{multline*}
	Hence, by compactness of $K$ and \eqref{prop:1:eq:4} 
	\[
		\sup_{\lambda \in K} |\hat{a}_{k+1} p(\lambda_{k+1})  - \hat{a}_k p(\lambda_k)| \leq 
		c \left| \frac{1}{\hat{a}_{k+1}} - \frac{1}{\hat{a}_{k}} \right|
	\]
	for a constant $c > 0$, which applied to every entry in $[\hat{C}^i_{k+1}(\lambda) - \hat{C}^i_k(\lambda)]$
	(see \eqref{eq:19}) implies \eqref{prop:1:eq:5}. The proof is complete.
\end{proof}

The following Lemma will be needed in several places of this work. The only technicality here
is the existence of uniform constants.
\begin{lemma} \label{prop:3}
	Assume
	\[
		\lim_{n \rightarrow \infty} a_n = \infty.
	\]
	If for $N$-periodic sequences $s$ and $z$
	\[
		\lim_{n \rightarrow \infty} |r_n a_n - a_{n-1} - s_n| = 
		\lim_{n \rightarrow \infty} |q_n a_n - b_n - z_n| = 0,
	\]
	then
	\begin{equation} \label{prop:3:eq:1}
		\lim_{n \to \infty} \left| \frac{a_{n-1}}{a_n} - r_n \right| = 0, \quad 
		\lim_{n \to \infty} \left| \frac{b_{n}}{a_n} - q_n \right| = 0.
	\end{equation}
	Moreover, for every $j \geq 1$ there is a constant $c>0$ such that for all $m > 0$
	\begin{equation} \label{prop:3:eq:2}
		\calV_N\left( \frac{1}{a_n} : n \geq m \right) + \calV_N\left( \frac{a_{n+j}}{a_n} : n \geq m \right) + 
		\calV_N \left( \frac{b_n}{a_n} : n \geq m \right) \leq c f_m,
	\end{equation}
	where
	\[
		f_m = 
		\calV_N\left( \frac{1}{a_n} : n  \geq m \right) + 
		\calV_N\left( r_n a_{n} - a_{n-1} : n \geq m \right) + 
		\calV_N\big(b_n - q_n a_n : n \geq m \big).
	\]
\end{lemma}
\begin{proof}
	Observe that the sequences $(r_{n+1} a_{n} - a_{n} : n \geq 0)$ and 
	$(b_n - q_n a_n : n \geq 0)$ are bounded. Since 
	\begin{equation} \label{eq:4}
		\frac{a_{n+1}}{a_n} - \frac{1}{r_{n+1}} = \frac{1}{r_{n+1} a_n} (r_{n+1} a_{n+1} - a_n), \quad
		\frac{b_n}{a_n} - q_n = \frac{1}{a_n} (b_n - q_n a_n),
	\end{equation}
	and the divergence of $a_n$ we have \eqref{prop:3:eq:1}.

	By \eqref{eq:4} and Proposition~\ref{prop:6}, we can estimate
	\begin{multline*}
		\calV_N\bigg(\frac{a_{n+1}}{a_n} : n \geq m \bigg) 
		 = \calV_N\bigg(\frac{a_{n+1}}{a_n} - \frac{1}{r_{n+1}} : n  \geq m \bigg) 
		 = \calV_N\bigg((r_{n+1} a_{n + 1} - a_n) \frac{1}{r_{n+1} a_n} : n  \geq m \bigg) \\
		 \leq
		c \calV_N\big(r_{n+1} a_{n+1} - a_n : n \geq m \big) + 
		c \max_{n \in \NN} |1/r_n| \calV_N\bigg(\frac{1}{a_n} : n  \geq m \bigg)
	\end{multline*}
	for 
	\[
		c = \max \left( \sup_{n \in \NN} |1/(r_{n+1} a_n)|,\ 
		\sup_{n \in \NN} |r_{n+1} a_{n+1} - a_n| \right).
	\]	
	Similarly, by Proposition~\ref{prop:6} for each $j \geq 1$
	\[
		\calV_N\bigg(\frac{a_{n+j}}{a_n} : n \geq m \bigg)
		\leq
		c
		\calV_N\bigg(\frac{a_{n+1}}{a_n} : n \geq m \bigg),
	\]
	where
	\[
		c = j \sup_{n \in \NN} \left( \frac{a_{n+1}}{a_n} \right)^j.
	\]
	Finally, by \eqref{eq:4} and Proposition~\ref{prop:6} we can estimate
	\begin{align*}
		\calV_N\bigg(\frac{b_n}{a_n} : n \geq m \bigg)
		& =
		\calV_N\bigg(\frac{b_n}{a_n} - q_n : n \geq m \bigg) \\
		& \leq
		c
		\calV_N\bigg(\frac{1}{a_n} : n \geq m \bigg)
		+
		c
		\calV_N\big(b_n - q_n a_n : n \geq m \big)
	\end{align*}
	for
	\[
		c = \max \left( \sup_{n \in \NN} |1/a_n|,\ \sup_{n \in \NN} |b_n - q a_n| \right).
	\]
\end{proof}

The following Proposition is the main result of this section.
\begin{proposition} \label{prop:2}
	Let us assume that for $N$-periodic sequences $s$ and $z$
	\begin{enumerate}[(a)]
		\item $\begin{aligned}[b]
		\lim_{n \to \infty} a_n = \infty,
		\end{aligned}$ \label{prop:2:a}
		
		\item $\begin{aligned}[b]
		\lim_{n \to \infty} |r_n a_n - a_{n-1} - s_n| = 0, \quad 
		\lim_{n \to \infty} |q_n a_n - b_n - z_n| = 0.
		\end{aligned}$ \label{prop:2:b}
	\end{enumerate}
	Fix a compact set $K \subset \CC$. Then for 
	\begin{equation} \label{prop:2:eq:2}
		C_n(\lambda) = a_{n + N - 1} \big(X_n(\lambda) - \gamma \Id\big)
	\end{equation}
	one has that for every $i$
	\[
		\calC_i(\cdot) = \lim_{k \rightarrow \infty} C_{kN+i}(\cdot)
	\]
	uniformly on $K$, where
	\begin{equation} \label{prop:2:eq:1}
		\calC_i(\lambda) = \alpha_{i-1} \calD_i + \alpha_{i-1} \hat{\calC}^i(\lambda)
	\end{equation}
	for
	\begin{equation} \label{prop:2:eq:4}
		\calD_i = 
		\sum_{j=0}^{N-1}
		\frac{1}{\alpha_{i+j}}
		\left[
			\prod_{m=j+1}^{N-1} \hat{\calB}^{i+m} (0)
		\right]
		\begin{pmatrix}
			0 & 0 \\
			s_{i+j} & z_{i+j}
		\end{pmatrix}
		\left[
			\prod_{m=0}^{j-1} \hat{\calB}^{i+m} (0)
		\right] 
	\end{equation}
	and for $\hat{\calC}^i$ and $\hat{\calB}^{j}$ defined in \eqref{prop:1:eq:1} and \eqref{prop:1:eq:3},
	respectively. Moreover, a constant $c>0$ such that for all $m > 0$
	\begin{equation} \label{prop:2:eq:3}
		\sum_{n = m}^\infty \sup_{\lambda \in K} \big\| C_{n+N}(\lambda) - C_n(\lambda) \big\| \leq c f_m,
	\end{equation}
	where
	\[
		f_m = \calV_N\left( \frac{1}{a_n} : n \geq m \right) + 
		\calV_N\left( r_n a_{n} - a_{n-1} : n \geq m \right) + 
		\calV_N\big(b_n - q_n a_n : n \geq m \big).
	\]
\end{proposition}
\begin{proof}
	Fix compact $K \subset \CC$.
	Let
	\[
		\lambda_n = \lambda \frac{\alpha_{n+N-1}}{a_{n+N-1}}.
	\]
	Let us observe that we have the following telescoping sum
	\begin{equation*}
		X_n(\lambda) =  
		\sum_{j=0}^{N-1} 
		\left[ \prod_{m=j+1}^{N-1} \hat{\calB}^{n+m}(\lambda_n) \right]
		\left( B_{n+j}(\lambda) - \hat{\calB}^{n+j}(\lambda_n) \right)
		\left[ \prod_{m=0}^{j-1} B_{n+m}(\lambda) \right]
		+ \hat{\calX}^n (\lambda_n).
	\end{equation*}
	Hence, by \eqref{prop:2:eq:2}
	\begin{multline} \label{eq:1}
		C_n(\lambda) = \sum_{j=0}^{N-1} 
		\left[ \prod_{m=j+1}^{N-1} \hat{\calB}^{n+m}(\lambda_n) \right]
		a_{n+N-1} \left( B_{n+j}(\lambda) - \hat{\calB}^{n+j}(\lambda_n) \right)
		\left[ \prod_{m=0}^{j-1} B_{n+m}(\lambda) \right] \\
		+ a_{n+N-1} \left( \hat{\calX}^n (\lambda_n) - \gamma \Id \right).
	\end{multline}
	For every $i$ we set
	\begin{equation}
		\calC_i(\lambda) = \lim_{k \to \infty} a_{(k+1)N+j-1} \big(X_{kN+i}(\lambda) - \gamma \Id \big) = 
		\lim_{k \to \infty} C_{kN+i}(\lambda).
	\end{equation}
	First, we show that the convergence is uniform on $K$. It is enough to analyse 
	each factor in \eqref{eq:1} separately. Let us consider $n = kN + i$. Lemma~\ref{prop:3} 
	implies that for every~$j$
	\[
		\lim_{k \rightarrow \infty} \hat{\calB}^{kN+i+j}(\lambda_{k N+i+j}) =
		\lim_{k \rightarrow \infty} B_{kN+i+j}(\lambda) =
		\begin{pmatrix}
			0 & 1 \\
			-r_{i+j} & -q_{i+j}
		\end{pmatrix}
	\]
	and the convergence is uniform on $K$. 
	Next, for every $j$
	\begin{multline} \label{eq:3}
		a_{n+N-1} \left( B_{n+j}(\lambda) - \hat{\calB}^{n+j}(\lambda_n) \right) = \\
		\frac{a_{n+N-1}}{a_{n+j}}
		\begin{pmatrix}
			0 & 0 \\
			r_{n+j} a_{n+j} - a_{n+j-1} & \lambda 
			\left( 1 - \frac{a_{n+j}}{a_{n+N-1}} \frac{\alpha_{n+N-1}}{\alpha_{n+j}} \right) + 
			q_n a_{n+j} - b_{n+j}
		\end{pmatrix},
	\end{multline}
	which implies
	\[
		\lim_{k \rightarrow \infty} a_{(k+1)N+i+j-1} 
		\left( B_{kN+i+j}(\lambda) - \hat{\calB}^{kN+i+j}(\lambda_{kN+i}) \right) =
		\frac{\alpha_{i-1}}{\alpha_{i+j}}
		\begin{pmatrix}
			0 & 0 \\
			s_{i+j} & z_{i+j}
		\end{pmatrix}
	\]
	uniformly on $K$. Finally, Proposition~\ref{prop:1} implies
	\[
		\lim_{k \rightarrow \infty} 
		a_{(k+1)N+i-1} \left( \hat{\calX}^{kN+i} (\lambda_{kN+i}) - \gamma \Id \right) =
		\hat{\calC}^i(\alpha_{i-1} \lambda) = \alpha_{i-1} \hat{\calC}^i(\lambda).
	\]
	It proves \eqref{prop:2:eq:1}.
	
	Let us turn to the proof of \eqref{prop:2:eq:3}. We now consider an algebra $C(K; M_{2}(\CC))$
	with the norm $\normM{\cdot}$ (see \eqref{eq:5}). We will use the total $N$-variation with respect 
	to this norm.
	
	As a consequence of Proposition~\ref{prop:6}, to estimate \eqref{prop:2:eq:3}, 
	it is enough to show that each factor in \eqref{eq:1} has the supremum and the total 
	$N$-variation bounded. 
	Every factor in \eqref{eq:1} is uniformly convergent, hence it has the supremum bounded.
	In view of Proposition~\ref{prop:6}, it is enough to	prove that there is $c > 0$ such that
	for all $j \in \{ 0, 1, \ldots, N-1 \}$
	\begin{gather}
		\label{eq:2a}
		\calV_N\big(B_n(\lambda) : n \geq m) \leq c f_m, \\
		\label{eq:2b}
		\calV_N\big(\hat{\calB}^{n+j}(\lambda_n) : n \geq m) \leq c f_m, \\		
		\label{eq:2c}
		\calV_N \left( a_{n+N-1} \big(B_{n+j}(\lambda) - \hat{\calB}^{n+j}(\lambda_n) \big) : 
		n \geq m \right) \leq c f_m, \\
		\label{eq:2d}
		\calV_N \left( a_{n+N-1} \left( \hat{\calX}^n (\lambda_n) - \gamma \Id \right) : 
		n \geq m \right) \leq c f_m.
	\end{gather}
	We have	
	\[
		\big\lVert B_{n+N}(\lambda) - B_n(\lambda) \big\rVert
		\leq
		\bigg|\frac{a_{n+N-1}}{a_{n+N}} - \frac{a_{n-1}}{a_n} \bigg|
		+
		\abs{\lambda}\bigg|\frac{1}{a_{n+N}} - \frac{1}{a_n} \bigg|
		+
		\bigg|\frac{b_{n+N}}{a_{n+N}} - \frac{b_n}{a_n} \bigg|,
	\]
	thus by Lemma~\ref{prop:3} and Proposition~\ref{prop:6} we get
	\eqref{eq:2a}. Similarly, for \eqref{eq:2b} we obtain
	\[
		\big\lVert \hat{\calB}^{n+N+j}(\lambda_n) - \hat{\calB}^{n+j}(\lambda_n) \big\rVert
		\leq
		\abs{\lambda} \alpha_{n+N-1} \bigg|\frac{1}{a_{n+2N-1}} - \frac{1}{a_{n+N-1}} \bigg|.
	\]
	Let us turn to the proof of the estimate \eqref{eq:2d}. For every $i \in \{0, 1, \ldots, N-1\}$ 
	we define
	\[	
		\hat{a}^i = \left( \frac{a_{(k+1)N+i-1}}{\alpha_{i-1}} : k \geq 0 \right).
	\]
	Let $(\hat{C}^{i}_n(\lambda) : n \geq 0)$ denote the corresponding sequence of the matrices 
	\eqref{prop:1:eq:2}. Then
	\[
		a_{(k+1)N+i-1} \left( \hat{\calX}^{kN+i} (\lambda_{kN+i}) - \gamma \Id \right)
		= \alpha_{i-1} \hat{C}^i_k(\lambda).
	\]
	Assumption~\eqref{prop:2:a} imply that for every $i$ the sequence $\tilde{a}^i$ satisfies
	the hypothesis of Proposition~\ref{prop:1}. By Proposition~\ref{prop:1},
	\[
		\sup_{\lambda \in K} \norm{\hat{C}^i_{k+1}(\lambda) - 
		\hat{C}^i_{k}(\lambda)} \leq 
		c_i 
		\left| \frac{1}{a_{(k+2) N + i - 1}} - \frac{1}{a_{(k+1) N + i - 1}} \right|
	\]
	for a constant $c_i > 0$. Hence,
	\begin{multline*}
		\calV_N \left( a_{n+N-1} \left( \hat{\calX}^n (\lambda_n) - \gamma \Id \right) : 
		n \geq m \right) = 
		\sum_{i= m \bmod{N}}^{N-1} \alpha_{i-1} \sup_{\lambda \in K} \norm{\hat{C}^i_{\floor{m/N}+1}(\lambda) 
		- \hat{C}^i_{\floor{m/N}}(\lambda)} \\
		+ \sum_{i=0}^{N-1} \sum_{k=\floor{m/N}+1}^\infty \alpha_{i-1} \sup_{\lambda \in K} 
		\| \tilde{C}^i_{k+1}(\lambda) - \tilde{C}^i_{k}(\lambda) \| 
		\leq c \calV_N \left( \frac{1}{a_n} : n \geq m \right)
	\end{multline*}
	for a constant $c>0$ independent of $m$.
	Finally, \eqref{eq:2c} is a consequence of \eqref{eq:3} applied to Lemma~\ref{prop:3} 
	and Proposition~\ref{prop:6}. The proof is complete.
\end{proof}

\subsection{Non-degeneracy of the quadratic forms}
Let $\Lambda \subset \RR$ be an open set. According to \cite{Swiderski2017}, we say that 
$\{ Q^\lambda : \lambda \in \Lambda \}$ (see \eqref{eq:45})
is \emph{uniformly non-degenerated} on $K \subset \Lambda$ if there are $c \geq 1$ and 
$M \geq 1$ such that for all $v \in \RR^2$, $\lambda \in K$ and $n \geq M$ 
\[
	c^{-1} \norm{v}^2 \leq \abs{Q_n^\lambda(v)} \leq c \norm{v}^2.
\]
We say that $\{Q^\lambda : \lambda \in \Lambda\}$ is \emph{almost uniformly non-degenerated} 
on $\Lambda$ if it is uniformly non-degenerated on each compact subset of $\Lambda$.

Observe that $\sprod{E v}{v} = 0$ (see \eqref{eq:45}). This simple observation lies behind the 
following non-degeneracy test, whose proof is exactly the same as the one given for 
\cite[Proposition 2]{Swiderski2017}. Notice that in \cite{Swiderski2017} an additional assumption 
was needed, because it was stated for a slightly different~$Q$.
\begin{proposition} \label{prop:5}
	Suppose that there is $\gamma \in \RR$ such that for all $j \in \{0, \ldots, N-1\}$
	\[
		\lim_{k \to \infty} a_{(k+1)N + j-1}(X_{k N + j}(\cdot) - \gamma \Id) = \calC_j(\cdot)
	\]
	uniformly on compact subsets of $\Lambda \subset \RR$. If for all $j \in \{0, \ldots, N-1\}$
	and all $\lambda \in \Lambda$
	\begin{equation} \label{prop:5:eq:1}
		\discr \calC_j(\lambda) < 0,
	\end{equation}
	then the family $\{Q^\lambda : \lambda \in \Lambda\}$ is almost uniformly non-degenerated on $\Lambda$.
\end{proposition}

We are ready to prove the main result of this section.
\begin{proposition} \label{prop:7}
	Let $\calC_i(\lambda)$ be defined in \eqref{prop:2:eq:1}. Let 
	\begin{equation} \label{prop:7:eq:1}
		h_i(\lambda) = -\discr[\calC_i(\lambda)], \quad (i=0,1,\ldots,N-1).
	\end{equation}
	Then for every $i$
	\begin{equation} \label{prop:7:eq:2}
		h_0(\lambda) = (r_0 r_1 \cdots r_{i-1})^2 h_i(\lambda).
	\end{equation}
	Consequently, the family $\big\{Q^\lambda : \lambda \in \Lambda \big\}$ is almost uniformly 
	non-degenerated for
	\begin{equation} \label{prop:7:eq:3}
		\Lambda = \{ \lambda \in \RR : h_0(\lambda) > 0 \}.
	\end{equation}
\end{proposition}
\begin{proof}
	To show \eqref{prop:7:eq:2} it is enough to prove that for every $i$
	\begin{equation} \label{eq:7}
		\frac{1}{\alpha_i} [\hat{\calB}^i(0)]^{-1} \calC_{i+1}(\lambda) \hat{\calB}^i(0) = 
		\frac{1}{\alpha_{i-1}} \calC_i(\lambda).
	\end{equation}
	
	We have
	\begin{equation} \label{eq:44}
		\frac{1}{\alpha_{i-1}}\calC_{i}(\lambda) = \calD_i + \hat{\calC}^{i}(\lambda).
	\end{equation}
	Let us begin with the first term. We have that 
	$[ \hat{\calB}^i (0) ]^{-1} \calD_i \hat{\calB}^i (0)$ equals
	\[
		\sum_{j=0}^{N-1}
		\frac{1}{\alpha_{i+j+1}}
		\left[ \hat{\calB}^i (0) \right]^{-1}
		\left[
			\prod_{m=j+1}^{N-1} \hat{\calB}^{i+m+1} (0)
		\right]
		\begin{pmatrix}
			0 & 0 \\
			s_{i+j+1} & z_{i+j+1}
		\end{pmatrix}
		\left[
			\prod_{m=0}^{j-1} \hat{\calB}^{i+m+1} (0)
		\right]
		\hat{\calB}^i (0).
	\]
	By shifting the indexing of $j$ and $m$ by $1$ we obtain that the last expression equals
	\begin{multline} \label{eq:8}
		\sum_{j=1}^{N-1}
		\frac{1}{\alpha_{i+j}}
		\left[
			\prod_{m=j+1}^{N-1} \hat{\calB}^{i+m} (0)
		\right]
		\begin{pmatrix}
			0 & 0 \\
			s_{i+j} & z_{i+j}
		\end{pmatrix}
		\left[
			\prod_{m=0}^{j-1} \hat{\calB}^{i+m} (0)
		\right]
		\\ + 
		\left[ \hat{\calB}^i (0) \right]^{-1}
		\begin{pmatrix}
			0 & 0 \\
			s_{i+N} & z_{i+N}
		\end{pmatrix}		
		\left[
			\prod_{m=0}^{N-1} \hat{\calB}^{i+m} (0)
		\right].
	\end{multline}
	Because
	\[
		\prod_{m=0}^{N-1} \hat{\calB}^{i+m} (0) = \gamma \Id
	\]
	it commutes with every matrix. Therefore, by \eqref{eq:8}
	\begin{equation} \label{eq:6}
		\left[ \hat{\calB}^i (0) \right]^{-1}
		\calD_{i+1}
		\hat{\calB}^i (0)  =
		\calD_i 
	\end{equation}
	
	Let us consider the second term of $\calC_{i}(\lambda)$. By $N$-periodicity of $\alpha$ and $\beta$ 
	we obtain
	\[
		 \hat{\calX}^i(\lambda) =
		[\hat{\calB}^i(\lambda)]^{-1} \hat{\calX}^{i+1}(\lambda) \hat{\calB}^i(\lambda).
	\]
	Hence,
	\[
		\hat{C}^{i}_k(\lambda) = 
		\left[ \hat{\calB}^i \left( \frac{\lambda}{\hat{a}_k} \right) \right]^{-1}
		\hat{C}^{i+1}_k(\lambda)
		\hat{\calB}^i \left( \frac{\lambda}{\hat{a}_k} \right).
	\]	
	By tending with $k$ to infinity, we obtain
	\[
		\hat{\calC}^{i}(\lambda) = 
		\left[ \hat{\calB}^i (0) \right]^{-1}
		\hat{C}^{i+1} (\lambda)
		\hat{\calB}^i (0).
	\]
	This combined with and \eqref{eq:44} and \eqref{eq:6} implies \eqref{eq:7}.
\end{proof}

Finally, for the further reference we collect the results of this section in a self-contained
form.
\begin{corollary}	\label{cor:2}
	Let $N$ be a positive integer. Suppose that $(r_n : n \geq 0)$ and $(q_n : n \geq 0)$ are
	$N$-periodic sequences such that
	\[
		\prod_{i=0}^{N-1} 
		\begin{pmatrix}
			0 & 1 \\
			-r_i & -q_i
		\end{pmatrix} 
		=
		\gamma \Id, \quad \gamma \in \{-1, 1 \}.
	\]
	Assume
	\begin{enumerate}[(a)]
		\item $\begin{aligned}[b]
			\lim_{n \rightarrow \infty} a_n = \infty,
		\end{aligned}$
		
		\item $\begin{aligned}[b]
			\lim_{n \rightarrow \infty} |r_n a_n - a_{n-1} - s_n| = 0,
		\end{aligned}$
		
		\item $\begin{aligned}[b]
			\lim_{n \rightarrow \infty} |b_n - q_n a_n - z_n| = 0
		\end{aligned}$
	\end{enumerate}
	for $N$-periodic sequences $s$ and $z$.
	Then the family $\{ Q^\lambda : \Lambda \}$ defined in \eqref{eq:45} is almost uniformly non-degenerated 
	for
	\[
		\Lambda = \RR \setminus I(s,z)
	\]
	where $I(s,z)$ is an explicitly computable compact interval depending on the sequences $s$ and 
	$z$. Moreover, if $s \equiv 0$ and $z \equiv 0$, then $I(s,z) = \{ 0 \}$.
\end{corollary}
\begin{proof}
	By Proposition~\ref{prop:7} we have to compute when $h_0(\lambda)$ is positive. We have
	\begin{equation} \label{eq:55}
		h_0(\lambda) = -\discr[\calC_0(\lambda)] = 4 \det[\calC_0(\lambda)] - (\tr[\calC_0(\lambda)])^2.
	\end{equation}
	
	Let us begin with the computation of the trace. By Proposition~\ref{prop:2} (see \eqref{prop:2:eq:1})
	\begin{equation} \label{eq:56}
		\calC_0(\lambda) = \alpha_{N-1} \calD_0 + \alpha_{N-1} \hat{\calC}^0(\lambda).
	\end{equation}
	Hence, by Proposition~\ref{prop:1} (see \eqref{prop:1:eq:1}) and Proposition~\ref{prop:10}
	\begin{equation} \label{eq:57}
		\tr[\calC_0(\lambda)] = \alpha_{N-1} \tr[\hat{\calC}^0(\lambda)] + \alpha_{N-1} \tr[\calD_0] 
		= \alpha_{N-1} \tr[\calD_0].
	\end{equation}
	
	Let us turn to the determinant. Observe that Proposition~\ref{prop:1} (see \eqref{prop:1:eq:1})
	implies that $\hat{\calC}^0(\lambda) = \lambda \hat{\calC}^0(1)$. Hence, using \eqref{eq:55} and 
	computing formally the determinant, we obtain
	\[
		\alpha_{N-1}^{-2} \det[\calC_0(\lambda)] = \det[\hat{\calC}^0(1)] \lambda^2 + c \lambda + 
		\det[\calD_0]
	\]
	for some constant $c \in \RR$.
	Therefore, by \eqref{eq:55}--\eqref{eq:57}
	\[
		\alpha_{N-1}^{-2} h_0(\lambda) = (4 \det[\hat{\calC}^0(1)]) \lambda^2 + (4 c) \lambda -
		\discr[\calD_0].
	\]
	Now, by Proposition~\ref{prop:10} we obtain that $\det[\hat{\calC}^0(1)] > 0$. Moreover, if
	$s \equiv 0$ and $z \equiv 0$, then $\calD_0 = 0$, and consequently,
	\[
		\alpha_{N-1}^{-2} h_0(\lambda) = \discr[\hat{\calC}^0(1)] \lambda^2,
	\]
	which by Proposition~\ref{prop:10} is strictly positive for $\lambda \neq 0$. The proof is complete.
\end{proof}

\subsection{The proof of the convergence}
The following lemma is an abstraction of the proof of \cite[Theorem 3]{Swiderski2017}.
\begin{lemma} \label{lem:2}
	Let 
	\begin{equation} \label{lem:2:eq:1}
		X_n = \gamma \Id + \frac{1}{a_{n+N-1}} C_n
	\end{equation}
	for a constant $\gamma \in \RR$. Then for every generalised eigenvector $u$ associated with $\lambda \in \RR$ 
	and $\alpha \in \RR^2 \setminus \{ 0 \}$
	\[
		S_n(\alpha, \lambda) = 		
		\frac{a_{n+N-1}^2}{a_{n-1}}
		\bigg\langle
		E C_n(\lambda)
		\begin{pmatrix}
			u_{n+N-1} \\
			u_{n+N}
		\end{pmatrix},
		\begin{pmatrix}
			u_{n+N-1} \\
			u_{n+N}
		\end{pmatrix}
		\bigg\rangle
	\]
	and
	\[
		| S_{n+N}(\alpha, \lambda) - S_n(\alpha, \lambda) | \leq 
		a_{n+N-1} \left\| \frac{a_{n+2N-1}}{a_{n+N-1}} C_{n+N}(\lambda) - 
		\frac{a_{n+N-1}}{a_{n-1}} C_n(\lambda) \right\| (u_{n+N-1}^2 + u_{n+N}^2),
	\]
	where $S_n$ is defined in \eqref{eq:17}.
\end{lemma}
\begin{proof}
	By \eqref{lem:2:eq:1}, we obtain
	\begin{align*}
		S_{n+N}(\alpha, \lambda) & = 
		a_{n+2N-1}^2 \bigg\langle E X_{n+N}(\lambda)
		\begin{pmatrix}
			u_{n+N-1} \\
			u_{n+N}
		\end{pmatrix},
		\begin{pmatrix}
			u_{n+N-1} \\
			u_{n+N}
	\end{pmatrix}
		\bigg\rangle \\
		& = a_{n+2N-1}
		\bigg\langle
		E C_{n+N}(\lambda)
		\begin{pmatrix}
			u_{n+N-1} \\
			u_{n+N}
		\end{pmatrix},
		\begin{pmatrix}
			u_{n+N-1} \\
			u_{n+N}
		\end{pmatrix}
		\bigg\rangle.
	\end{align*}
	Moreover,
	\begin{align*}
   	S_n(\alpha, \lambda) & = a_{n+N-1}^2
		\bigg\langle 
		E 
		\begin{pmatrix}
			u_{n+N-1} \\
			u_{n+N}
		\end{pmatrix},
		X^{-1}_n(\lambda)
		\begin{pmatrix}
			u_{n+N-1} \\
			u_{n+N}
		\end{pmatrix}
		\bigg\rangle \\
		& =
		a_{n+N-1}^2 
		\bigg\langle
		\big[ X^{-1}_n(\lambda) \big]^* E
		\begin{pmatrix}
			u_{n+N-1} \\
			u_{n+N}
		\end{pmatrix},
		\begin{pmatrix}
			u_{n+N-1} \\
			u_{n+N}
		\end{pmatrix}
		\bigg\rangle.
	\end{align*}
	Observe that for every invertible matrix $X \in M_{2}(\mathbb{R})$ we have 
	\[
      \frac{1}{\det X} E X = [X^{-1}]^* E.
	\]
	Therefore,
	\begin{align*}
		S_n(\alpha, \lambda) & = \frac{a_{n+N-1}^3}{a_{n-1}} 
		\bigg\langle
   		E X_n(\lambda)
		\begin{pmatrix}
			u_{n+N-1} \\
			u_{n+N}
		\end{pmatrix},
		\begin{pmatrix}
			u_{n+N-1} \\
			u_{n+N}
		\end{pmatrix}
		\bigg\rangle \\
		& =
		\frac{a_{n+N-1}^2}{a_{n-1}}
		\bigg\langle
		E C_n(\lambda)
		\begin{pmatrix}
			u_{n+N-1} \\
			u_{n+N}
		\end{pmatrix},
		\begin{pmatrix}
			u_{n+N-1} \\
			u_{n+N}
		\end{pmatrix}
		\bigg\rangle
	\end{align*}
	and by Schwarz inequality the result follows.
\end{proof}

Finally, we are ready to prove Theorem~\ref{thm:C}. The techniques are the same as used in the proof of
\cite[Theorem 3]{Swiderski2017}.
\begin{proof}[Proof of Theorem~\ref{thm:C}]
	We fix a compact interval $I \subset \Lambda$, a compact connected set
	$\Omega \subset \RR^2 \setminus \{ 0 \}$ and 
	we consider a sequence of functions $(S_n : n \in \NN)$ defined by \eqref{eq:17}. 
	By Proposition~\ref{prop:7}, the family $\big\{Q^\lambda : \lambda \in I \big\}$ is uniformly 
	non-degenerated. \cite[Theorem 1]{Swiderski2017} implies that we have to show that there
	are $c \geq 1$ and $M > 0$ such that
	\begin{equation} \label{eq:59}
		c^{-1} \leq \abs{S_n(\alpha, \lambda)} \leq c
	\end{equation}
	for all $\alpha \in \Omega$, $\lambda \in I$ and $n > M$. To do so, similarly as in the proof of 
	\cite[Theorem 3]{Swiderski2017}, it is enough to show that
	\begin{equation} \label{eq:60}
		\sum_{n = M}^\infty \sup_{\alpha \in \Omega} \sup_{\lambda \in I} \abs{F_n(\alpha, \lambda)} 
		< \infty,
	\end{equation}
	where $(F_n : n \geq M)$ is a sequence of functions on $\Omega \times I$ 
	defined by
	\[
		F_n = \frac{S_{n + N} - S_n}{S_n}.
	\]
	Indeed,
	\begin{equation} \label{eq:61}
		\prod_{j = 0}^{k-1} (1 + F_{jN+M}) = \prod_{j = 0}^{k-1} \frac{S_{(j+1) N+M}}{S_{j N+M}} 
		= \frac{S_{kN+M}}{S_M}
	\end{equation}
	and the condition \eqref{eq:60} implies that the product \eqref{eq:61} is convergent uniformly
	on $\Omega \times I$ do a continuous function of definite sign. This implies \eqref{eq:59}.

	Since $\big\{Q^\lambda : \lambda \in I\big\}$ is uniformly non-degenerated, we have
	\[
		\abs{S_n(\alpha, \lambda)}
		\geq
		c^{-1}
		\frac{a_{n+N-1}^2}{a_{n-1}} 
		\big(u_{n+N-1}^2 + u_{n+N}^2\big)
	\]
	for all $n \geq M$, $\alpha \in \Omega$ and $\lambda \in I$. Hence, by Lemma~\ref{lem:2}
	\begin{equation} \label{eq:18}
		\big|F_n(\alpha, \lambda) \big|
		\leq
		c 
		\frac{a_{n-1}}{a_{n+N-1}} 
		\bigg\lVert \frac{a_{n+2N-1}}{a_{n+N-1}} C_{n+N}(\lambda) - 
		\frac{a_{n+N-1}}{a_{n-1}} C_n(\lambda) \bigg\rVert,
	\end{equation}
	which, by Proposition~\ref{prop:2} (see \eqref{prop:2:eq:3}), Lemma~\ref{prop:3} and 
	Proposition~\ref{prop:6}, is uniformly summable with respect to $\alpha \in \Omega$ and $\lambda \in I$.
\end{proof}
\begin{corollary} \label{cor:1}
	Under the hypothesis of Theorem \ref{thm:C}, for each $j \in \{0, \ldots, N-1\}$ the subsequence of
	continuous functions $(S_{kN+j} : k \in \NN)$ converges uniformly on compact subsets of
	$(\RR^2 \setminus \{ 0 \} ) \times \Lambda$ to the function $\widetilde{g}^j$ without zeros.
	Moreover, by \eqref{eq:18} for every compact $\Omega \subset \RR^2 \setminus \{ 0 \}$ and a compact 
	interval $I \subset \Lambda$ there is a constant $c>0$ such that for every $m \equiv j \Mod{N}$
	\begin{multline*}
		\sup_{\alpha \in \Omega} \sup_{\lambda \in I} 
		|\widetilde{g}^j(\alpha, \lambda) - S_m(\alpha, \lambda)|
		\leq c \calV_N\bigg(\frac{1}{a_n} : n \geq m \bigg) \\
		+ c \calV_N(a_n - a_{n-1} : n \geq m) + c \calV_N(b_n - q a_n : n \geq m).
	\end{multline*}
\end{corollary}

\section{A formula for density} \label{sec:formDens}
The aim of this section is to prove Theorem~\ref{thm:D}. In fact, it is rather a corollary of the
convergence of Turán determinants proven in Section~\ref{sec:asympGenEig}. Once we have the convergence,
the rest of our proof is very similar to \cite[Theorem 2]{Swiderski2017b}. Since the proof from
\cite{Swiderski2017b} uses various estimates obtained in the proofs (rather than stated explicitly)
of several results, it is easier to present a self-contained proof than to list the necessary changes. 
Moreover, the proof included here is more detailed.

For a positive integer $N$ we define $N$-shifted Turán determinants by
\[
	D^N_n(x) = p_n(x) p_{n+N-1}(x) - p_{n-1}(x) p_{n+N}(x).
\]
Then define
\[
	S_n(x) = a_{n+N-1}^2 D^N_n(x).
\]
By \eqref{eq:17} one has
\[
	S_n(x) = S_n(\alpha(x), x), \quad \alpha(x) = (1, p_1(x)).
\]
In particular, by Corollary~\ref{cor:1}, for every $i$ the limit defining $\widetilde{g}^i(x)$ exists 
for $x \in \Lambda$ and for every compact $K \subset \Lambda$ one has
\[
	\lim_{k \rightarrow \infty} \sup_{x \in K} |S_{kN+i}(x) - \widetilde{g}^i(x)| = 0
\]
with the upper bound on the rate of convergence. Thus, we obtained the first part of Theorem~\ref{thm:D}.
It remains to compute the value of $\widetilde{g}^i(x)$, which is done below.
\begin{proof}[Proof of Theorem~\ref{thm:D}]
	We proceed similarly to the method used in the proof of Theorem~\cite[Theorem 2]{Swiderski2017b}. 
	
	For any $K \geq 0$ consider the "truncated" sequences $a^K$ and $b^K$, i.e. defined by
	\begin{equation} \label{eq:27}
		a^K_{n} = a_n, \quad b^K_{n} = b_n \quad (n \leq K+N-1)
	\end{equation}
	and extended $N$-periodically by
	\begin{equation} \label{eq:28}
		a^K_{K+kN+i} = a_{K+i}, \quad b^K_{K+kN+i} = b_{K+i} \quad 
		(k \geq 1,\ i \in \{ 0, 1, \ldots, N-1 \}).
	\end{equation}
	Let $(p^K_n : n \geq 0)$ be the corresponding sequence of orthonormal polynomials associated with
	$a^K$ and $b^K$. Let $\mu_K$ the the corresponding orthonormalizing measure for $p^K$.
	By \eqref{eq:27} we have that 
	\begin{equation} \label{eq:29}
		p_n(x) = p^K_n(x), \quad (n \leq N+K).
	\end{equation}
	
	Since the Carleman condition for $A$ is satisfied, the measure $\mu$ is determined by its moments. 
	By \eqref{eq:29} the moments of $\mu_K$ and $\mu$ coincide up to order $K+N$.
	Consequently, we have the weak convergence of the sequence 
	$(\mu_K : K \geq 0)$ to the measure $\mu$ (see, e.g. \cite[Theorem 30.2]{Billingsley1995}), 
	i.e. for every compact interval $I \subset \RR$
	\[
		\lim_{K \rightarrow \infty} \mu_K(I) = \mu(I).
	\]
	Our aim is to show that for every $i \in \{0, 1, \ldots, N-1 \}$ and every compact interval 
	$I \subset \Lambda$ one has
	\begin{equation} \label{eq:25}
		\lim_{k \rightarrow \infty} \mu_{kN+i}(I) = \widetilde{\mu}_i(I).
	\end{equation}
	for the Borel measure
	\[
		\widetilde{\mu}_i(S) = \int_{S \cap \Lambda} \frac{\sqrt{h_i(x)}}{2 \pi \widetilde{g}^i(x)} \ud x,
		\quad (S \in \Bor(\RR)).
	\]
	Consequently, the sequence $(\mu_{kN+i}\restriction_\Lambda : k \geq 0)$ converges weakly to 
	$\widetilde{\mu}_i$ (see, e.g. \cite[Example 2.3]{Billingsley1999}). Then by the uniqueness of the 
	weak limit we get $\mu\restriction_\Lambda = \widetilde{\mu}_i$ for every $i$ and the proof is complete. 
	By Lebesgue Dominated Convergence Theorem equation \eqref{eq:25} will be satisfied if the measures 
	$\mu_{kN+i}$ are absolutely continuous on $I$ for sufficiently large $k$ and
	\begin{equation} \label{eq:26}
		\lim_{k \rightarrow \infty} 
		\sup_{x \in I} \left| \mu'_{kN+i}(x) - \frac{\sqrt{h_i(x)}}{2 \pi \widetilde{g}^i(x)} \right| = 0.
	\end{equation}
	 
	It remains to show \eqref{eq:26}. Let $I \subset \Lambda$ be a compact interval. Let
	\begin{equation} \label{eq:31}
		X_n^K(x) = \prod_{j=n}^{n+N-1} B_{j}^K(x),
	\end{equation}
	where $B_j^K$ is the transfer matrix associated with sequences $a^K$ and $b^K$. By \eqref{eq:28} 
	we have $X^K_{K+kN} = X^K_{K+N}$ for $k \geq 1$. Let
	\begin{equation} \label{eq:36}
		\Lambda_K = \{ x \in \RR : \discr[X^K_{K+N}(x)] < 0 \}.
	\end{equation}
	Then by \cite[Theorem 3]{Swiderski2017b} (for the original formulation see 
	\cite[Theorem 6]{GeronimoVanAssche1991})
	\begin{equation} \label{eq:40}
		\widetilde{g}^i_{K}(x) = 
		\lim_{k \rightarrow \infty} |S^K_{K+kN}(x)|, \quad 
		x \in \Lambda_K, \quad (K \equiv i \Mod{N})
	\end{equation}
	exists and defines a continuous positive function. Moreover, the measure $\mu_K$ is absolutely
	continuous on $\Lambda_K$ and its density is equal to
	\begin{equation} \label{eq:33}
		\mu'_{K}(x) = \frac{a_{K+N-1} \sqrt{-\discr[X^K_{K+N}(x)]}}{2 \pi \widetilde{g}^i_K}, \quad 
		x \in \Lambda_K, \quad (K \equiv i \Mod{N}).
	\end{equation}
	
	Consider the decomposition
	\begin{equation} \label{eq:30}
		X_K(x) = \gamma \Id + \frac{1}{a_{K+N-1}} C_K(x), \quad 
		X^K_{K+N}(x) = \gamma \Id + \frac{1}{a_{K+N-1}} C^K_{K+N}(x)
	\end{equation}
	for some matrices $C_K$ and $C^K_{K+N}$, where $X_K$ and $X^K_{K+N}$ are defined in \eqref{eq:32} 
	and \eqref{eq:31}, respectively. Let us show that
	\begin{equation} \label{eq:34}
		\lim_{K \rightarrow \infty} \norm{C^K_{K+N}(\cdot) - C_K(\cdot)} = 0
	\end{equation}
	uniformly on compact subsets of $\RR$.
	From \eqref{eq:28} we get
	\begin{equation} \label{eq:35}
		C^K_{K+N}(x) - C_K(x) = \frac{a_{K+N-1}}{a_K} \left[ \prod_{j=K+1}^{K+N-1} B_j(x) \right] 
		\begin{pmatrix}
			0 & 0 \\
			a_{K-1} - a_{K+N-1} & 0
		\end{pmatrix}.
	\end{equation}	 
	Lemma~\ref{prop:3} implies that for every $j$ the sequence $(B_{mN+j}(x) : m \geq 0)$
	is uniformly convergent (hence bounded) on compact subsets of $\RR$. Therefore, by the 
	assumption~\eqref{thm:D:eq:1} and Lemma~\ref{prop:3} we obtain \eqref{eq:34}.
	 
	Hence, by \eqref{eq:33}, \eqref{eq:30} and \eqref{part3:eq:4}
	\begin{equation} \label{eq:37}
		\mu'_K(x) = \frac{\sqrt{-\discr[C^K_{K+N}(x)]}}{2 \pi \widetilde{g}^i_K(x)}, \quad 
		x \in \Lambda_K, \quad (K \equiv i \Mod{N}).
	\end{equation}
	By \eqref{eq:34}, Proposition~\ref{prop:2} we have uniform convergence on compact subsets of $\RR$
	\begin{equation} \label{eq:38}
		\lim_{k \rightarrow \infty} \sqrt{-\discr[C^{kN+i}_{kN+i+N}(x)]} = 
		\lim_{k \rightarrow \infty} \sqrt{-\discr[C_{kN+i}(x)]} =
		\sqrt{\calC_i(x)} = 
		\sqrt{h_i(x)},
	\end{equation}
	where $h_i$ is defined in \eqref{prop:7:eq:1}. Hence, by \eqref{eq:36} and Proposition~\ref{prop:7} 
	there exists $K_0$ such that for every $K \geq K_0$ one has $I \subset \Lambda_K$. Therefore,
	by \eqref{eq:37}, \eqref{eq:38} and Corollary~\ref{cor:1} to prove \eqref{eq:26} it is enough to show that
	\begin{equation} \label{eq:39}
		\lim_{k \rightarrow \infty} \widetilde{g}^i_{kN+i} = \widetilde{g}^i
	\end{equation}
	uniformly on $I$.
	 
	Let us turn to the proof of \eqref{eq:39}. We consider $K \equiv i \Mod{N}$ for $0 \leq i < N$.
	By Lemma~\ref{lem:2}
	\begin{multline} \label{eq:41}
		|S^K_{n+N}(x) - S^K_n(x)| \leq 
		\bigg\lVert a^K_{n+2N-1} C^K_{n+N}(x) - \frac{(a^K_{n+N-1})^2}{a^K_{n-1}} C^K_n(x) \bigg\rVert \\
		\times \left[ (p^K_{n+N-1}(x))^2 + (p^K_{n+N}(x))^2 \right].
	\end{multline}
	Therefore, by \eqref{eq:28}, $S^K_{n+N} = S^K_n$ for $n \geq K+1$. Take $n=K+N$, then 
	$\widetilde{g}^i_K(x) = S^K_{K+N}(x)$. Observe, that formula \eqref{eq:28} force that
	\[
		X^K_{K}(x) = X_{K}(x), \quad C^K_K(x) = C_K(x),
	\]
	and consequently, also $S^K_{K}(x) = S_{K}(x)$.
	Then \eqref{eq:41} and \eqref{eq:29} implies
	\begin{multline} \label{eq:42}
		|S^K_{K+N}(x) - S_K(x)| \leq a_{K+N-1} \left( 
		\norm{C^K_{K+N}(x) - C_K(x)} + \left| 1 - \frac{a_{K+N-1}}{a_{K-1}} \right| \norm{C_K(x)} \right) \\
		\times \left[ p^2_{K+N-1}(x) + p^2_{K+N}(x) \right].
	\end{multline}
	By the similar reasoning we also have that
	\begin{multline} \label{eq:43}
		|S_{K+N}(x) - S_K(x)| \leq a_{K+N-1} \left( 
		\norm{C_{K+N}(x) - C_K(x)} + \left| 1 - \frac{a_{K+N-1}}{a_{K-1}} \right| \norm{C_K(x)} \right) \\
		\times \left[ p^2_{K+N-1}(x) + p^2_{K+N}(x) \right].
	\end{multline}
	By Lemma~\ref{lem:2} and Proposition~\ref{prop:7}
	\[
		|S_K(x)| \geq c \frac{a^2_{K+N-1}}{a_{K-1}} \left[ p^2_{K+N-1}(x) + p^2_{K+N}(x) \right]
	\]
	for a constant $c>0$. Therefore, by \eqref{eq:34}, \eqref{eq:42} and \eqref{eq:43} both
	\[
		F_K(x) := \frac{S_{K+N}(x) - S_K(x)}{S_K(x)}, \quad F^K_K(x) := \frac{S^K_{K+N}(x) - S_K(x)}{S_K(x)}
	\]
	tend to $0$ uniformly on $I$ as long as $\norm{C_K(x)}$ is uniformly bounded on $I$. 
	It is the case, since by Proposition~\ref{prop:2} the sequence $(C_{kN+i}(x) : k \geq 0)$ is uniformly
	convergent on $I$ for every $i$. Therefore,
	\[
		\frac{S^K_{K+N}(x)}{S_{K+N}(x)} = \frac{1 + F^K_K(x)}{1 + F_K(x)}
	\]
	tends to $1$ uniformly. Since
	\[
		\widetilde{g}^i_K = |S^K_{K+N}(x)| = \frac{|1 + F^K_K(x)|}{|1 + F_K(x)|} |S_{K+N}(x)|
	\]
	and the definition of $\widetilde{g}^i$, we obtain
	\[
		\lim_{k \rightarrow \infty} \widetilde{g}^i_{kN+N+i}(x) = \widetilde{g}^i(x)
	\]
	uniformly on $I$. It shows \eqref{eq:39}. The proof is complete.
   \end{proof}

\section{Examples} \label{sec:examples}
\subsection{Construction of the modulating sequences} \label{sec:constrModulSeq}
	In this section we are interested in methods of construction sequences $\alpha$ and $\beta$
	such that 
	\begin{equation} \label{eq:10}
		\calF(0) := 
		\prod_{j=1}^N 
		\begin{pmatrix}
			0 & 1 \\
			-\frac{\alpha_{j-1}}{\alpha_{j}} & -\frac{\beta_j}{\alpha_j}
		\end{pmatrix}
		= \gamma \Id, \quad \gamma \in \{-1, 1 \}.
	\end{equation}
	It is of primary importance since the results of this paper are closely related
	to such sequences.
	
\subsubsection{Explicit methods}
Let us begin with two specific examples.
\begin{example}
	Let $N=2M$, $\alpha$ be a positive $N$-periodic sequence and $\beta_n \equiv 0$. Then
	\eqref{eq:10} is satisfied if
	\begin{equation} \label{eq:13}
		\alpha_0 \alpha_2 \ldots \alpha_{N-2} = \alpha_1 \alpha_3 \ldots \alpha_{N-1}.
	\end{equation}
	In such a case $\gamma = (-1)^M$. Moreover, if $M$ is an odd integer and sequence $\alpha$ is in fact 
	$M$ periodic, then \eqref{eq:13} is automatically satisfied.
\end{example}
\begin{proof}
	One has
	\[
		\begin{pmatrix}
			0    & 1 \\
			-\frac{\alpha_j}{\alpha_{j+1}} & 0
		\end{pmatrix}
		\begin{pmatrix}
			0  & 1 \\
			-\frac{\alpha_{j-1}}{\alpha_{j}} & 0
		\end{pmatrix}
		=
		\begin{pmatrix}
			-\frac{\alpha_{j-1}}{\alpha_{j}} & 0 \\
			0 & -\frac{\alpha_j}{\alpha_{j+1}}
		\end{pmatrix}.
	\]
	Hence,
	\[
		\prod_{j=1}^{N}
		\begin{pmatrix}
			0  & 1 \\
			-\frac{\alpha_{j-1}}{\alpha_{j}} & 0
		\end{pmatrix}
		= (-1)^M
		\begin{pmatrix}
			\frac{\alpha_1 \alpha_3 \ldots \alpha_{N-1}}{\alpha_0 \alpha_2 \ldots \alpha_{N-2}} & 0 \\
			0 & \frac{\alpha_0 \alpha_2 \ldots \alpha_{N-2}}{\alpha_1 \alpha_3 \ldots \alpha_{N-1}}
		\end{pmatrix}.
	\]
	Therefore, \eqref{eq:10} holds as long as \eqref{eq:13} is satisfied.
\end{proof}	
	
\begin{example}
	Let $N \geq 2$ and
	\[
		\alpha_n \equiv 1, \quad  \beta_n \equiv 2 \cos \frac{k_0 \pi}{N} \quad 
		(k_0 \in \{ 1, 2, \ldots, N-1 \}).
	\]
	Then \eqref{eq:10} is satisfied for $\gamma = (-1)^{N+k_0}$ (see \cite[Section 3.3]{Swiderski2017}).
\end{example}

To construct more examples we need additionally tools.
\begin{proposition} \label{prop:12}
	Fix $M$. Let $\alpha$ and $\beta$ be $M$-periodic sequences such that
	\[
		\tr[\calF(0)] = 0, \quad \text{for} \quad
		\calF(0) := 
		\prod_{j=1}^M 
		\begin{pmatrix}
			0 & 1 \\
			-\frac{\alpha_{j-1}}{\alpha_{j}} & -\frac{\beta_j}{\alpha_j}
		\end{pmatrix}.
	\]
	Then \eqref{eq:10} is satisfied for $\alpha$ and $\beta$ with $N=2M$ and $\gamma = -1$.
\end{proposition}
\begin{proof}
	Let $d = \frac{\alpha_{M-1}}{\alpha_0}$. By Proposition~\ref{prop:13}, we have
	\[
		\calF(0) =
		\begin{pmatrix}
			-d w_{M-2}^{[1]}(0) & w_{M-1}(0) \\
			-d w_{M-1}^{[1]}(0) & w_{M}(0)
		\end{pmatrix}.
	\]     
	Observe that
	\[
		\calF^2(0) = 
		\begin{pmatrix}
			d^2 (w_{M-2}^{[1]}(0))^2 - d w_{M-1}^{[1]}(0) w_{M-1}(0) & w_{M-1}(0) \tr[\calF(0)] \\
			-d w_{M-1}^{[1]}(0) \tr[\calF(0)] & w_M^2(0) - d w_{M-1}^{[1]}(0) w_{M-1}(0)
		\end{pmatrix}.
	\]
	Proposition~\ref{prop:11} implies
	\[
		d (w_{M-1}(0) w_{M-1}^{[1]}(0) - w_{M}(0) w_{M-2}^{[1]}(0)) = 1.
	\]
	By $\tr[\calF(0)] = 0$, we have
	\[
		w_M(0) = d w_{M-2}^{[1]}(0).
	\]
	Hence,
	\[ 
		\calF^2(0) = -\Id
	\]
	and the result follows.
\end{proof}

Below we illustrate the usefulness of Proposition~\ref{prop:12} on some examples.
\begin{example}
	Let $\alpha_n = (n \mod 2) + 1$ and $\beta_n = q$. Then
	\[
		\tr \Bigg[ \prod_{j=1}^{2} \hat{\calB}^j(0) \Bigg] = \frac{1}{2} q^2 - \frac{5}{2}.
	\]
	Hence, the assumptions of Proposition~\ref{prop:12} are satisfied for $q \in \{ -\sqrt{5}, \sqrt{5} \}$.
\end{example}

\begin{example}
	Let $\alpha_n = (n \mod 3) + 1$ and $\beta_n = q$. Then
	\[
		\tr \Bigg[ \prod_{j=1}^{3} \hat{\calB}^j(0) \Bigg] = -\frac{1}{6} q^3 + \frac{7}{3} q.
	\]
	Hence, the assumptions of Proposition~\ref{prop:12} are satisfied for 
	$q \in \{ -\sqrt{14}, 0, \sqrt{14} \}$.
\end{example}

\begin{example}
	Let $\alpha_n = (n \mod 4) + 1$ and $\beta_n = q$. Then
	\[
		\tr \Bigg[ \prod_{j=1}^{4} \hat{\calB}^j(0) \Bigg] = 
		\frac{1}{24} q^4 - \frac{5}{4} q^2 + \frac{73}{24}.
	\]
	Hence, the assumptions of Proposition~\ref{prop:12} are satisfied for
	\[
		q \in 
		\left\{ 
			-\sqrt{15 + 2 \sqrt{38}}, 
			-\sqrt{15 - 2 \sqrt{38}}, 
			\sqrt{15 - 2 \sqrt{38}}, 
			\sqrt{15 + 2 \sqrt{38}} 
		\right\}.
	\]
\end{example}

Finally, the following example is of different type than the previous ones.
\begin{example}
	Let $\alpha_n \equiv 1$, $\beta_{2k} = (-1)^k$ and $\beta_{2k+1} = 0$. Then
	\[
		\prod_{j=1}^{4} \hat{\calB}^j(0) = \Id.
	\]
\end{example}

\subsubsection{Geometric interpretation}
Let $(\alpha_n: n \geq 0)$ and $(\beta_n : n \geq 0)$ be $M$-periodic. Let us denote
\begin{equation} \label{eq:48}
	\calF(x) = \prod_{j=1}^M 
	\begin{pmatrix}
		0 & 1 \\
		-\frac{\alpha_{j-1}}{\alpha_j} & \frac{x - \beta_n}{\alpha_n}
	\end{pmatrix}.
\end{equation}
Let $A_\per$ be the Jacobi matrix associated with the sequences $\alpha$ and $\beta$. Then one has
\begin{equation} \label{eq:50}
	(\tr \calF)^{-1}[(-2,2)] = \bigcup_{i=1}^M I_i,
\end{equation}
where $I_i$ are open non-empty disjoint intervals. Moreover, 
\begin{equation} \label{eq:62}
	\sigmaEss{A_\per} = (\tr \calF)^{-1}[[-2,2]]
\end{equation}
and the matrix $A_\per$ is purely absolutely continuous on $\sigmaEss{A_\per}$ (see \cite[Chapter 5]{Simon2010}).

By the continuity of $\tr \calF$ one has that in every $I_i$ there is at least one $x_i \in I_i$ such 
that $\tr \calF(x_i) = 0$. On the other hand, $\tr \calF(x)$ is a polynomial of degree $M$.
Consequently, in every $I_i$ is exactly one zero of $\tr \calF$.

Let $x_0$ such that $\tr \calF(x_0)=0$. Then by considering the sequence $(\beta_n - x_0 : n \geq 0)$ 
we can assume that $x_0=0$. Hence, by the fact that $\tr \calF$ have exactly $M$ real simple zeros, 
we always have  a rich family of sequences $\alpha$ and $\beta$ satisfying the assumptions of 
Proposition~\ref{prop:12}.

The next Proposition has been proven in \cite[Proposition 5.4.3]{Simon2010}.
\begin{proposition} \label{prop:17}
	Suppose that $|\tr \calF(x_0)| = 2$ for some $x_0 \in \RR$. Then $\calF(x_0) = \gamma \Id$ if and only if 
	the point $x_0$ lies on the boundary of two distinct intervals from the set $\{I_i : i = 1,2,\ldots,M \}$.
\end{proposition}

Proposition~\ref{prop:17} together with \eqref{eq:50} and \eqref{eq:62} enables us to provide a geometric 
interpretation of the properties of $\calF$. Hence, motivated by the random matrix theory 
(see \cite{Lubinsky2016}), we can use the following terminology concerning the periodic perturbations 
of unbounded Jacobi matrices.
\begin{definition} \label{def:2}
Let $M$ be a positive integer. Assume
\begin{enumerate}[(a)]
	\item 
	$\begin{aligned}[b]
		\lim_{n \to \infty} a_n = \infty,
	\end{aligned}$
	\item 
	$\begin{aligned}[b]
		\lim_{n \to \infty} \Big| \frac{a_{n-1}}{a_n} - r_n \Big| = 0,
	\end{aligned}$
	\item 
	$\begin{aligned}[b]
		\lim_{n \to \infty} \Big| \frac{b_n}{a_n} - q_n \Big| = 0
	\end{aligned}$
\end{enumerate}
for $M$-periodic sequences $r$ and $q$. Suppose that $r_0 r_1 \cdots r_{M-1} = 1$. Let the sequences 
$\alpha$ and $\beta$ be such that
\[
	r_n = \frac{\alpha_{n-1}}{\alpha_n}, \qquad q_n = \frac{\beta_n}{\alpha_n}
\] 
(see Proposition~\ref{prop:4}) and let $\calF$ be defined by the formula \eqref{eq:48}. 
Then we say that the Jacobi matrix $A$ corresponds to 
\begin{enumerate}[(a)]
	\item the bulk regime if $|\tr \calF(0)| < 2$,
	\item the soft edge regime if $\calF(0) = \pm \Id$,
	\item the hard edge regime if $\calF(0)$ is not diagonalisable.
\end{enumerate}
\end{definition}

Let us emphasise that the parameter $M$ is an integral part of Definition~\ref{def:2}. 
The integer $M$ need not to be the \emph{minimal} period of the sequences $\alpha$ and $\beta$. 
Indeed, Proposition~\ref{prop:12} shows that a Jacobi matrix can correspond to the bulk or 
to the soft edge regime depending only on the choice of $M$. On the other hand, Proposition~\ref{prop:17} 
implies that the hard edge regime is invariant on the choice of $M$.

According to Definition~\ref{def:2}, Theorems~\ref{thm:C} and \ref{thm:D} correspond to the soft edge regime 
(see Lemma~\ref{prop:3}), whereas Theorems~\ref{thm:A} and \ref{thm:B} correspond to the soft edge regime provided
\[
	\lim_{k \to \infty} \tilde{a}_k = \infty, \qquad 
	\lim_{k \to \infty} \frac{\tilde{a}_{k-1}}{\tilde{a}_k} = 1.
\]

\subsection{Multiple weights} \label{sec:ex:multiple}
Let us begin with the class of sequences examined in Theorem~\ref{thm:A} and \ref{thm:B}. This class
has been investigated before in the literature for sequences $\alpha_n \equiv 1$ and $\beta_n \equiv 0$. 
More specifically, some special cases were directly examined in 
\cite{DombrowskiPedersen2002a, DombrowskiPedersen2002, Moszynski2003} and in an equivalent model with
the spectrum on the half-line in \cite{DombrowskiPedersen1995, HintonLewis1978, Sahbani2008}. 

The following Example for the constant modulating sequences has been examined in \cite{Swiderski2017}.

\begin{example} \label{example:1}
	Fix a positive integer $N$. Let $\tilde{a}$ be a sequence of positive numbers such that
	\begin{enumerate}[(a)]
		\item $\begin{aligned}
			\calV_1(\tilde{a}_n - \tilde{a}_{n-1} : n \geq 1) + 
			\calV_1\bigg(\frac{1}{\tilde{a}_n} : n \geq 0 \bigg) < \infty;
		\end{aligned}$ \label{example:1:eq:1}
		\item $\begin{aligned}
			\lim_{n \to \infty} \tilde{a}_n = \infty;
		\end{aligned}$ \label{example:1:eq:2}
		\item $\begin{aligned}
			\lim_{n \to \infty} (\tilde{a}_n - \tilde{a}_{n-1}) = 0.
		\end{aligned}$ \label{example:1:eq:3}
	\end{enumerate}
	Let $(\alpha_n : n \geq 0)$ and $(\beta_n : n \geq 0)$ be $N$-periodic sequences such that 
	$\alpha_n > 0$ for every $n$ and \eqref{eq:10} is satisfied.
	Set
	\[
		a_{kN+i} = \alpha_i \tilde{a}_k, \quad 
		b_{kN+i} = \beta_i \tilde{a}_k, \quad 
		(i=0,1, \ldots N-1;\ k \geq 0).
	\]
	Then we have $\sigma(A) = \RR$ and the spectrum is purely absolutely continuous.
\end{example}
\begin{proof}
	Set
	\[
		r_n = \frac{\alpha_{n-1}}{\alpha_n}, \quad 
		q_n = \frac{\beta_{n}}{\alpha_n}.
	\]
	
	Let us begin by showing that condition \eqref{thm:C:eq:2} from Theorem~\ref{thm:C} is satisfied.
	Observe that
	\begin{equation} \label{eq:46}
		r_{i+1} a_{lN+i+1} = 
		\begin{cases}
			\alpha_{i} \tilde{a}_{l+1} & \text{for } 0 \leq i < N-1, \\
			\alpha_{i} \tilde{a}_{l+2} & \text{for } i = N-1.
		\end{cases}
	\end{equation}
	We have
	\[
		\calV_N ( r_{n+1} a_{n+1} - a_n : n \geq 0) = 
		\sum_{k=0}^\infty \sum_{i=0}^{N-1} 
		| (r_{i+1} a_{(k+1)N+i+1} - a_{(k+1)N+i}) - (r_{i+1} a_{kN+i+1} - a_{kN+i}) |,
	\]
	which by \eqref{eq:46} equals
	\[
		\sum_{k=0}^\infty \alpha_{N-1}| (\tilde{a}_{k+2} - \tilde{a}_{k+1}) - (\tilde{a}_{k+1} - \tilde{a}_k) | =
		\alpha_{N-1} \calV_1 (\tilde{a}_{k+1} - \tilde{a}_k : k \geq 0) < \infty.
	\]
	Next, since $b_n - q_n a_n \equiv 0$ we have
	\[
		\calV_N(b_n - q_n a_n : n \geq 0) = 0 < \infty.
	\]
	Finally, by Proposition~\ref{prop:6}\eqref{prop:6b}
	\[
		\calV_N \left(\frac{1}{\tilde{a}_n} : n \geq 0 \right) \leq N 
		\calV_1 \left(\frac{1}{\tilde{a}_n} : n \geq 0 \right) < \infty,
	\]
	which by Proposition~\ref{prop:6}\eqref{prop:6d} implies condition \eqref{thm:C:eq:2} 
	from Theorem~\ref{thm:C}.
	
	By assumption~\eqref{example:1:eq:3} we have $s_n \equiv 0$ and $z_n \equiv 0$. Hence, by
	Corollary~\ref{cor:2}, the assumptions of Theorem~\ref{thm:D} are satisfied for 
	$\Lambda = \RR \setminus \{ 0 \}$. By Theorem~\ref{thm:A} the point $0$ is not an eigenvalue.
	The proof is complete.
\end{proof}

\subsection{Periodic modulations} \label{sec:ex:mod}
The class of sequences considered in this section has been introduced in \cite{JanasNaboko2002}, where
under stronger regularity assumptions was examined the case when $\tr \calF(0) \neq \pm 2$ (see \eqref{eq:10}).

In \cite{Damanik2007, Motyka2014, Motyka2015, Naboko2009, Simonov2007} were considered some special cases 
corresponding to the more challenging setting of the hard edge regime (see Definition~\ref{def:2}).

The following Example for the constant modulating sequences is covered by the results obtained 
in \cite{Swiderski2017}.

\begin{example} \label{example:2}
	Let $N$ be a positive integer. Let $\tilde{a}$ be a positive sequence such that
	\begin{enumerate}[(a)]
		\item $\begin{aligned}
			\calV_N\bigg(\frac{1}{\tilde{a}_n} : n \geq 0 \bigg) +
			\calV_N\bigg(\tilde{a}_n - \tilde{a}_{n-1} : n \geq 1 \bigg) < \infty;
		\end{aligned}$ \label{example:2:eq:1}
		\item $\begin{aligned}
			\lim_{n \to \infty} \tilde{a}_n = \infty;
		\end{aligned}$ \label{example:2:eq:2}
		\item $\begin{aligned}
			\lim_{n \to \infty} (\tilde{a}_n - \tilde{a}_{n-1}) = 0.
		\end{aligned}$ \label{example:2:eq:3}
	\end{enumerate}
	Let $(\alpha_n : n \in \ZZ)$ and $(\beta_n : n \in \ZZ)$ be $N$-periodic sequences such that 
	$\alpha_n > 0$ for every $n$ and \eqref{eq:10} is satisfied. Set 
	\[
		a_n = \alpha_n \tilde{a}_n \quad\text{and} \quad b_n = \beta_n \tilde{a}_n.
	\]
	Then we have $\sigma(A) = \RR$ and the spectrum of the matrix~$A$ is absolutely continuous on 
	$\RR \setminus \{ 0 \}$.
\end{example}
\begin{proof}
	Set
	\[
		r_n = \frac{\alpha_{n-1}}{\alpha_n}, \quad 
		q_n = \frac{\beta_n}{\alpha_n}.
	\]
	By the choice of $\alpha$ and $\beta$ condition \eqref{thm:C:eq:1} from Theorem~\ref{thm:C} is satisfied.
	Then we have
	\[
		r_n a_n  - a_{n-1} = \alpha_{n-1} (\tilde{a}_n - \tilde{a}_{n-1}), \quad
		b_n - q_n a_n = 0, \quad
		\frac{1}{a_n} = \frac{1}{\alpha_n} \frac{1}{\tilde{a}_n}.
	\]
	By Proposition~\ref{prop:6}\eqref{prop:6d} and \eqref{example:2:eq:1} it implies that condition 
	\eqref{thm:C:eq:2} from Theorem~\ref{thm:C} is satisfied. Moreover, by \eqref{example:2:eq:3}
	it implies that $s \equiv 0$ and $z \equiv 0$, and consequently, by Corrolary~\ref{cor:2} 
	$\Lambda = \RR \setminus \{ 0 \}$. Condition Theorem~\ref{thm:C}\eqref{thm:C:eq:3}
	is satisfied by \eqref{example:2:eq:2}. The proof is complete.
\end{proof}

\subsection{Additive perturbations} \label{sec:ex:additive}
Finally, let us consider the additive periodic perturbations of sequences satisfying the assumptions of 
Theorem~\ref{thm:C}. The case when $\alpha_n \equiv 1$, $\beta_n \equiv 0$ and $\tilde{a}$ being some
regular sequence (usually $\tilde{a}_n = (n+1)^\alpha$ for $\alpha \in (0,1]$) has been examined extensively,
see \cite{Dombrowski2004, Dombrowski2009, DombrowskiJanasMoszynskiEtAl2004, DombrowskiPedersen2002a, DombrowskiPedersen2002, JanasMoszynski2003, JanasNabokoStolz2004, Sahbani2016}.

The following Example for constant $r$ and $q$ is covered by the results from \cite{Swiderski2017}.

\begin{example} \label{example:addPert}
	Fix a positive integer $N$ and $N$-periodic sequences $r$ and $q$ such that
	\[
		\calF = \prod_{j = 1}^N
		\begin{pmatrix}
			0 & 1 \\
			-r_j & -q_j 
		\end{pmatrix} = \gamma \Id, \quad \gamma \in \{-1, 1 \}.
	\]
	Let $d$ and $d'$ be $N$-periodic sequences and let $\tilde{a}$ and $\tilde{b}$ be sequences of positive 
	and real numbers, respectively. Assume
	\begin{enumerate}[(a)]
		\item $\begin{aligned}
			\calV_N(r_n \tilde{a}_n - \tilde{a}_{n-1} : n \geq 1) + 
			\calV_N(\tilde{b}_n - q_n \tilde{a}_n : n \geq 0) + 
			\calV_N\bigg(\frac{1}{\tilde{a}_n} : n \geq 0 \bigg) 
			< \infty.
		\end{aligned}$ \label{example:addPert:eq:1}
		\item $\begin{aligned}
			\lim_{n \to \infty} \tilde{a}_n = \infty;
		\end{aligned}$ \label{example:addPert:eq:2}
	\end{enumerate}
	Let
	\begin{equation} \label{example:addPert:eq:3}
		\lim_{n \rightarrow \infty} |r_n \tilde{a}_n - \tilde{a}_{n-1} - \tilde{s}_n| = 0, \quad 
		\lim_{n \rightarrow \infty} |q_n \tilde{a}_n - \tilde{b}_n - \tilde{z}_n| = 0
	\end{equation}
	for $N$-periodic sequences $\tilde{s}$ and $\tilde{z}$.
	Let
	\[
		a_n = \tilde{a}_n + d_n > 0 \quad\text{and}\quad b_n =  \tilde{b}_n + d'_n.
	\]
	Then $\sigma(A) \supseteq \overline{\Lambda}$ and the spectrum is absolutely continuous on
	$\Lambda$, where $\Lambda$ is defined in \eqref{prop:7:eq:3} for
	\begin{equation} \label{example:addPert:eq:4}
		s_n = \tilde{s}_n + r_n d_n - d_{n-1}, \quad z_n = \tilde{z}_n + d_n' - q_n d_n.
	\end{equation}
\end{example}
\begin{proof}
	Observe
	\[
		\frac{a_n}{\tilde{a}_n} = 1 + \frac{d_n}{\tilde{a}_n}.
	\]
	Therefore, by assertions \eqref{prop:6c} and \eqref{prop:6e} from Proposition~\ref{prop:6}, we obtain
	\[
		\calV_N\bigg(\frac{\tilde{a}_n}{a_n} : n \geq 0 \bigg) < \infty.
	\]
	Since
	\[
		\frac{1}{a_n} = \frac{\tilde{a}_n}{a_n} \frac{1}{\tilde{a}_n},
	\]
	Proposition~\ref{prop:6}\eqref{prop:6d} with assumption~\eqref{example:addPert:eq:1}
	implies
	\begin{equation} \label{eq:47}
		\calV_N\bigg(\frac{1}{a_n} : n \geq 0 \bigg) < \infty.
	\end{equation}
	
	We have
	\begin{align*}
		r_n a_n - a_{n-1}& = (r_n \tilde{a}_n - \tilde{a}_{n-1}) + r_n d_n - d_{n-1}, \\
		b_n - q_n a_n &= (\tilde{b}_n - q_n \tilde{a}_n) + d'_n - q_n d_n.
	\end{align*}
	Hence, by Proposition~\ref{prop:6}\eqref{prop:6c}, assumption \eqref{example:addPert:eq:1}
	and \eqref{eq:47} we obtain \eqref{thm:C:eq:2} from Theorem~\ref{thm:C}. 
	Furthermore, by \eqref{example:addPert:eq:3}, we also have \eqref{example:addPert:eq:4}. 
	The proof is complete.
\end{proof}

Example~\ref{example:addPert} shows that the sequences satisfying assumptions of Theorem~\ref{thm:C} is closed
under additive periodic perturbations but the position and the size of the gap in $\sigmaEss{A}$
may depend on the perturbation. 

\bibliographystyle{abbrv} 
\bibliography{critical}
\end{document}